\pdfoutput=1
\documentclass{article}

\usepackage{geometry}
\usepackage{graphicx}
\usepackage{color}
\usepackage{subfig}
\usepackage{float}
\usepackage{multirow}
\usepackage{algorithm}
\usepackage{algpseudocode}
\usepackage{mathtools}
\usepackage{mathptmx}      

\usepackage{amssymb,amsmath,amsthm}
\usepackage{enumitem}

\usepackage{url}
\newtheorem{theorem}{Theorem}[section]
\newtheorem{lemma}{Lemma}[section]
\newtheorem{proposition}{Proposition}[section]
\newtheorem{corollary}{Corollary}[section]
\newtheorem{definition}{Definition}[section]
\newtheorem*{remark}{Remark}

\DeclareMathOperator{\intS}{int}

\DeclareMathOperator*{\argmin}{arg\,min}

\DeclareMathOperator{\bdry}{bdry}

\makeatletter
\renewcommand*\env@matrix[1][\arraystretch]{%
  \edef\arraystretch{#1}%
  \hskip -\arraycolsep
  \let\@ifnextchar\new@ifnextchar
  \array{*\c@MaxMatrixCols c}}
\makeatother

\providecommand{\Lmat}{L\text{-matrix}}

\newcommand{\DUAL}[1]{{#1}^{D}}
\providecommand{\R}{\mathbb{R}}

\newcommand{\MYK}[1]{#1^{k}}
\newcommand{\MYKp}[1]{#1^{k+1}}

\newcommand{\NC}[2]{N_{#1}#2}

\newcommand{\reds}[1]{{{\tilde{#1}}}}

\mathtoolsset{showonlyrefs=true}

\begin{document}

\newcommand{\im}{\operatorname{im}\,}
\newcommand{\ri}{\operatorname{ri}\,}
\newcommand{\itr}{\operatorname{int}}
\newcommand{\pspace}{\operatorname{par}\,}
\newcommand{\rec}{\operatorname{rec}\,}
\newcommand{\lin}{\operatorname{lin}\,}
\newcommand{\rank}{\operatorname{rank}}
\newcommand{\sspan}{\operatorname{span}\,}
\newcommand{\aff}{\operatorname{aff}\,}
\newcommand{\tr}{T}

\renewcommand{\algorithmicrequire}{\textbf{Input:}}
\renewcommand{\algorithmicensure}{\textbf{Output:}}
\algrenewcommand\Return{\State \algorithmicreturn{} }

\iffalse
\def\PATHAVI{\textsc{PathAVI}}
\def\PATH{\textsc{Path}}
\def\AVI{\textsc{AVI}}
\def\VI{\textsc{VI}}
\def\MCP{\textsc{MCP}}
\def\LCP{\textsc{LCP}}
\def\GE{\textsc{GE}}
\def\PL{\textsc{PL}}
\def\QR{\textsc{QR}}
\def\NNF{\textsc{NNF}}
\else
\def\PATHAVI{\textsc{PathAVI}}
\def\PATH{\textsc{Path}}
\def\AVI{\text{AVI}}
\def\VI{\text{VI}}
\def\MCP{\text{MCP}}
\def\LCP{\text{LCP}}
\def\GE{\text{GE}}
\def\PL{\text{PL}}
\def\QR{\text{QR}}
\def\NNF{\text{NNF}}
\fi

\title{A Structure-Preserving Pivotal Method\\ for Affine Variational Inequalities
}


\author{Youngdae~Kim
  \thanks{Wisconsin Institute for Discovery and Department of Computer
    Sciences, University of Wisconsin-Madison, 1210 West Dayton St., Madison,
    WI, 53706 \newline\newline
    Y. Kim       \newline Email: youngdae@cs.wisc.edu \newline
    O. Huber     \newline Email: ohuber2@wisc.edu \newline
    M. C. Ferris \newline Email: ferris@cs.wisc.edu}
  \and Olivier~Huber \footnotemark[1]
  \and Michael~C.~Ferris \footnotemark[1]\\
}



\maketitle

\begin{abstract}
Affine variational inequalities (AVI) are an important problem class that
generalize systems of linear equations, linear complementarity problems and
optimality conditions for quadratic programs. 
This paper describes \PATHAVI{}, a structure-preserving pivotal approach, that
can process (solve or determine infeasible) large-scale sparse instances of
the problem efficiently, with theoretical guarantees and at high accuracy.
\PATHAVI{} implements a strategy that is known to process models with good
theoretical properties without reducing the problem to specialized forms,
since such reductions may destroy structure in the models and can lead to
very long computational times. We demonstrate formally that \PATHAVI{}
implicitly follows the theoretically sound iteration paths, and 
can be implemented in a large scale setting using existing
sparse linear algebra and linear programming techniques without employing a
reduction. We also extend the class of problems that \PATHAVI{} can process.
The paper demonstrates the effectiveness of our approach by comparison to the \PATH{} solver used on a complementarity
reformulation of the AVI in the context of applications in friction contact
and Nash Equilibria problems. \PATHAVI{} is a general
purpose solver, and freely available under the same conditions as \PATH{}.

\end{abstract}
\newpage
\section{Introduction}
\label{sec:intro}

In this paper, we present \PATHAVI{}, a structure-preserving pivotal
method for affine variational inequalities (\AVI{}s) in $\mathbb{R}^n$.
An $\AVI(C,q,M)$ is defined as follows: given a polyhedral convex set $C$,
find $z \in C$ such that
\begin{equation}
\left\langle Mz + q, y - z\right\rangle \ge 0, \quad \forall y \in C, \tag{AVI}
\label{eq:avi}
\end{equation}
where $M\in\R^{n\times n}$, $q\in\R^{n}$ and $\langle \cdot, \cdot\rangle$ is the
usual Euclidean inner product. 
An \AVI{} is a linear generalized equations \cite{robinson79}
and we refer to~\cite{facchinei03} for results on existence, uniqueness, and stability
theory for such systems.

\PATHAVI{} tries to solve an $\AVI(C,q,M)$ by computing a zero of the normal
map~\cite{robinson92} associated with the \AVI{}. The normal map
$M_C:\mathbb{R}^n \rightarrow \mathbb{R}^n$ is defined as follows:
\begin{equation*}
M_C(x) \coloneqq M(\pi_C(x)) + q + x - \pi_C(x),
\tag{normal map}
\end{equation*}
with $\pi_C(\cdot)$ denoting the Euclidean projector onto the set $C$. One can
easily see that $M_C(x^*)=0$ if and only if $z^*=\pi_C(x^*)$ with
$x^*=z^*-(Mz^*+q)$ is a solution to the $\AVI(C,q,M)$. 
To compute a zero of $M_C(x)$, our method employs the complementary pivoting
method \cite{eaves76,lemke65} with a ray start: the piecewise-linear (\PL{}) map
$G_C: \mathbb{R}^n \times \mathbb{R}_+ \rightarrow \mathbb{R}^n$ is defined as
\begin{equation}
 G_C(x,t)\coloneqq M_C(x)-tr,
\label{eq:gc}
\end{equation}
with $r \in \mathbb{R}^n$ denoting the covering vector and $t$ the auxiliary
variable.
A path defined as $G_C^{-1}(0)$ is followed through complementary
pivoting. The algorithm terminates when either $t$ becomes zero (a solution to
the \AVI{} is found) or a secondary ray is generated. Under some additional
assumptions this latter outcome can be interpreted in terms of feasibility of
the \AVI{}.

This approach has been previously investigated in~\cite{cao96}. However
the method in~\cite{cao96} requires a reduction transforming the given 
$\AVI(C,q,M)$ to a reduced $\AVI(\tilde{C},\tilde{q},\tilde{M})$
to eliminate lines in $C$ and solves the reduced \AVI{}. The matrix
$\tilde{M}$ is constructed from a Schur complement computation and the 
polyhedral constraints defining $\tilde{C}$ are computed by multiplying
with orthonormal matrices. Thus the original structure in $C$ and $M$ may be
lost:
in particular if the \AVI{} is sparse, there is no guarantee that
the resulting reduced \AVI{} would enjoy the same property.
We provide an instance where this happens in Section~\ref{subsec:exp-preserve}.
In sharp contrast, \PATHAVI{} does not require any reduction at all. 
Therefore our method is able to take advantage of a sparse structure, whereas
the method in~\cite{cao96} often needs to perform dense linear algebra
computations.

The main challenge in tackling the problem in its original space lies in the
starting phase. For good theoretical properties, a ray start is required, and it is
well-defined at an extreme point. However, when $C$ contains lines there
is no extreme point.
To perform a ray start in that case, we need to find an \emph{implicit extreme point}, 
which generalizes the notion of an extreme point when the underlying feasible
region contains lines.
Roughly speaking, if we project an implicit extreme point of $C$ on the subset
where all lines are removed, we find an extreme point.
We show that there is an implicit extreme point satisfying the sufficient
conditions for a ray start.
We explain how the phase 1 of the simplex method can be used to find such a
point.

Regarding processability, we show that \PATHAVI{} can process an
\AVI{}$(C, q, M)$ whenever $M$ is an $L$-matrix with respect to the recession
cone of $C$ \cite[Definition 4.2]{cao96}.
We also exhibit two new classes of \AVI{} where \PATHAVI{} finds a solution.
The first one stems from the study of friction contact problems from an \AVI{} perspective,
and the second one can be seen as a generalization of a known existence result
for LCP for copositive matrices.
In contrast with the previous results, the conditions are on both $M$ and $q$.

One of the practical and most widely used method for solving an~\AVI{} has
been to use the \PATH{} solver~\cite{dirkse95}, which is one of the most
robust and efficient solvers for mixed complementarity problems (\MCP{}s).
It is well known~\cite{dirkse97,facchinei03} that an \AVI{} can be
reformulated as a linear \MCP{}, and \PATH{} uses this approach
when it solves an \AVI{}. However the \MCP{} reformulation does not exploit
the polyhedral structure of the set $C$, in that complementary pivoting
of \PATH{} is done over a different \PL{}-manifold from \PATHAVI{}'s.
We compare theoretical properties of the two formulations,
and present computational results comparing performance of the solvers.

This paper is organized as follows: in Section~\ref{sec:background}, we
briefly describe  how one uses the complementary pivoting method on
a \PL{}-manifold to compute a zero of the normal map associated with a given
\AVI{}. Section~\ref{sec:theory} presents our main theoretical results:
firstly, we discuss sufficient conditions for a ray start,
we define an implicit extreme point, and prove the existence of an implicit
extreme point satisfying the conditions for a ray start.
Secondly, we show that \PATHAVI{} can process $L$-matrices
and we show new types of \AVI{}s processable by \PATHAVI{}.
In Section~\ref{sec:computing}, we present the computational procedure to
start \PATHAVI{}. Section~\ref{sec:worst-analysis} introduces the \MCP{}
reformulation of the \AVI{} and analyzes worst-case performance of the
two formulations.
Finally we present computational results in Section~\ref{sec:exp}, and
Section~\ref{sec:conclusion} concludes this paper.

A word about our notation is in order. Let $S$ be a convex set in
$\mathbb{R}^n$. The lineality space of $S$ is denoted by $\lin S$.
The symbol $\ri S$ denotes the relative interior of $S$.
The affine hull of $S$ is denoted by $\aff S$. By $\pspace S$, we mean
the subspace parallel to $\aff S$ such that $\aff S = s + \pspace S$ for each
$s \in S$. When ordered index sets are used as subscripts on a matrix,
they define a submatrix: for ordered index sets
$\alpha \subset \{1,\dots,m\}$ and
$\beta \subset \{1,\dots,n\}$ $M_{\alpha \beta}$ denotes a submatrix of $M$
consisting of rows and columns of $M$ in the order of
$\alpha$ and $\beta$, respectively.
When matrices are used as subscripts on a matrix, they define another matrix:
for matrices $Q$ and $\bar{Q}$ having appropriate dimensions $M_{Q\bar{Q}}$
denotes $Q^TM\bar{Q}$. For an $\AVI(C,q,M)$, the set $C$ is assumed to be the
set $\{z \in \mathbb{R}^n \mid Az - b \in K, l\le z \le u\}$
with $l_j,u_j \in \mathbb{R} \cup \{-\infty,\infty\}, b_i \in \mathbb{R}$
and $A_{i\bullet} \neq 0$ for $i=1,\dots,m$ and $j=1,\dots,n$, where
the set $K$ is a Cartesian product of $\mathbb{R}_+$, $\{0\}$, or
$\mathbb{R}_-$ to accommodate constraints of the form $\geq$, $=$, or $\leq$,
respectively. For a closed convex cone $K$, the dual cone of $K$ is denoted
by $K^D := \{y \mid \langle y,k \rangle \ge 0, \forall k \in K\}$. 
For the rest of this paper, $Q$ and $\bar{Q}$ denote orthonormal basis
matrices for the lineality space of $C$ and its orthogonal complement, 
respectively. 

\section{Background}\label{sec:background}

In this section, we briefly describe how to compute a zero of the normal map
associated with a given $\AVI(C,q,M)$ using the complementary pivoting method
with a ray start. 
We also introduce some concepts related to processability of \AVI{}s.
Refer to~\cite{cao96,eaves76,lemke65,robinson92} for more details.

The basic procedure of the complementary pivoting method to compute a zero of
the normal map associated with an $\AVI(C,q,M)$ is as follows:
i) compute an initial solution $(x^0,t^0)$ such that $G_C(x^0,t^0)=0$,
and the point $(x^0,t^0)$ lies on a ray, called a starting ray, 
consisting of points $(x(t),t)$
with $G_C(x(t),t)=0$ and $\pi_C(x(t))=\pi_C(x^0)$ for all $t \ge t^0$;
then ii) starting from $(x^0,t^0)$ follow a path 
$G^{-1}(0)=\{(x,t) \in \mathbb{R}^n \times \mathbb{R}_+\mid G(x,t)=0\}$
using the complementary pivoting method until $t$ becomes zero or a
secondary ray is generated. As we will see, \PATHAVI{} generates a
starting ray at an implicit extreme point of $C$, i.e., $\pi_C(x^0)$
is an implicit extreme point.

Computationally, finding an initial solution $(x^0,t^0)$ amounts to
computing a complementary basic solution having $z=\pi_C(x^0)$
for the following system of equations:
\begin{equation}
  \begin{aligned}
    Mz + q - A^\tr \lambda - w + v &&=&& 0,\\
    Az - b &&=&& s,
  \end{aligned}
  \label{eq:comp-eq}
\end{equation}
with complementarity between variables
\begin{equation}
  \begin{aligned}
    K \ni s && \perp &&& \lambda \in K^D,\\
    0 \le z - l && \perp &&& w \ge 0,\\
    0 \le u - z && \perp &&& v \ge 0.
  \end{aligned}
  \label{eq:comp-perp}
\end{equation}
The complementary basic solution satisfies the sufficient conditions 
for a ray start as defined in Section~\ref{sec:theory}.
Then by adding $-tr$ with $r \in \ri(N_C(\pi_C(x^0)))$ to the first equation in
\eqref{eq:comp-eq} and pivoting in the $t$ variable, 
we generate an almost complementary feasible basis and start complementary
pivoting.

Geometrically, the map $G_C(x,t)$ is defined over a \PL{}$(n+1)$-manifold
$\mathcal{M}_C$, where definition of a manifold follows from 
\cite[Section 4]{eaves76}. The manifold $\mathcal{M}_C$ consists of
a pair $\left(\mathbb{R}^n\times\mathbb{R}_+,
\{\sigma_i \times \mathbb{R}_+ \mid i \in \mathcal{I}\}\right)$ such that
each $\sigma_i$ is a set formed by $\sigma_i = F_i + N_{F_i}$, where
$F_i$ is from a collection of the nonempty faces 
$\{F_i \mid i \in \mathcal{I}\}$ of $C$, and
$N_{F_i}$ is a normal cone having constant value on $\ri F_i$.
The manifold $\mathcal{M}_C$ is constructed from the normal manifold
$\mathcal{N}_C$ consisting of a pair
$\left(\mathbb{R}^n,\{\sigma_i\mid i \in \mathcal{I}\}\right)$ by doing
a Cartesian product each $\sigma_i$ with $\mathbb{R}_+$. Note that the
collection of the sets $\{\sigma_i \mid i \in \mathcal{I}\}$ is a subdivision of
$\mathbb{R}^n$. Consequently,
$\{\sigma_i\times\mathbb{R}_+ \mid i \in \mathcal{I}\}$ is a subdivision
of $\mathbb{R}^n \times \mathbb{R}_+$.
The $k$-dimensional faces of the $\sigma_i \times \mathbb{R}_+$ are called
the $k$-cells of $\mathcal{M}_C$. Similarly, the $k$-dimensional faces of
the $\sigma_i$ are called the $k$-cells of $\mathcal{N}_C$.
The map $G_C$ coincides with some affine transformation on each $(n+1)$-cell
$\sigma_i \times \mathbb{R}_+$ as the normal map $M_C$ does on each 
$n$-cell $\sigma_i$ \cite[Proposition 2.5]{robinson92}. 
Note that the starting ray $(x(t),t)$ for $t \ge t^0$
lies interior to some $(n+1)$-cell $\sigma_i \times \mathbb{R}_+$ of 
$\mathcal{M}_C$, where $x^0$ is a regular point, i.e.,
$\dim(G_C(\sigma_i \times \mathbb{R}_+))=n$.
Under lexicographic pivoting, each complementary pivoting generates each
piece of the 1-manifold $G^{-1}(0)$ such that it starts from a boundary
of a $(n+1)$-cell of $\mathcal{M}_C$ (except for the first piece containing
the starting ray) and passes through interior to that
cell until it reaches the cell's another boundary. If it does not reach
a boundary, then we say that a secondary ray is generated.
The set of $(n+1)$-cells the
1-manifold passes through never repeats. As there is a finite number of
$(n+1)$-cells of $\mathcal{M}_C$, we have either $t$ reaches zero
(equivalently we find a solution to the $\AVI(C,q,M)$) or a secondary ray is
generated.

Processability is tied to the conditions under which a secondary ray occurs.
As with the LCPs, the answer to this question involves specific
matrix classes that we now define.

%
\begin{definition}[Definition 4.1 \cite{cao96}]
Let $K$ be a closed convex cone. A matrix $M$ is said to be \emph{copositive} with respect to $K$ if $\langle x, Mx\rangle \geq 0$ for all $x\in K$.
 If furthermore it holds that for all $x\in K$ $\langle x, Mx\rangle = $ implies $(M + M^{\tr})x = 0$, then $M$ is \emph{copositive-plus}.
\end{definition}
\begin{definition}
 Let $K$ be a closed convex cone. A matrix $M$ is said to be \emph{semimonotone} with respect to $K$ if
 for every $q\in\ri(\DUAL{K})$, the solution set of the generalized complementarity problem
    \begin{equation}
      z\in K,\qquad Mz + q\in\DUAL{K}, \qquad z^\tr(Mz + q) = 0 \label{eq:Msemimonotone}
    \end{equation}
      is contained in $\lin K$.
\end{definition}
\begin{remark}
 This definition is consistent with the existing semimonotone property in the \LCP{}, as given in~\cite[Definition~3.9.1]{cottle2009linear}.
 In this case $K=\R^n_+$ and $\lin K = \{0\}$. Then the condition~\eqref{eq:Msemimonotone} is equivalent to $0$ being the solution set of \LCP{}$(M,q)$
for all $q > 0$, which by Theorem~3.9.3 in~\cite{cottle2009linear} is equivalent to the standard definition of $M$ semimonotone.
\end{remark}
\begin{definition}[Definition 4.2 \cite{cao96}]\label{def:Lmat}
Let $K$ be a closed convex cone. A matrix $M$ is said to be an $\Lmat$ with respect to $K$ if both
\begin{enumerate}
 \item[(a)] $M$ is semimonotone with respect to $K$
    \item[(b)] For any $z\neq 0$ satisfying
      \begin{equation}
        z\in K, \qquad Mz\in\DUAL{K}, \qquad z^\tr Mz = 0,
      \end{equation}
      there exists $z'\neq 0$ such that $z'$ is contained in every face of $K$ containing $z$ and $-M^\tr z'$ is contained in every face of $\DUAL{K}$ containing $Mz$.
\end{enumerate}
\end{definition}
\begin{lemma}[Lemma 4.5~\cite{cao96}]
  If a matrix $M$ is copositive-plus with respect to a closed convex cone $K$, then it is an $\Lmat$ with respect to $K$.
\end{lemma}
The main existing result on the processability using a pivotal method is the following.
\begin{theorem}[Theorem 4.4~\cite{cao96}]
Suppose that $C$ is a polyhedral convex set, and $M$ is an $L$-matrix with
respect to $\rec C$ which is invertible on the lineality space of $C$. Then
exactly one of the following occurs:
\begin{itemize}
\item[$\bullet$] The method of \cite{cao96} solves the $\AVI(C,q,M)$.
\item[$\bullet$] The following system has no solution
  \begin{equation*}
    Mz + q \in (\rec C)^D.
  \end{equation*}
\end{itemize}
\label{thm:cao}
\end{theorem}

\section{Theoretical results}\label{sec:theoretical-results}
\label{sec:theory}

In this section, we show that an implementation of \PATHAVI{} in the original
space enjoys the same properties as Theorem \ref{thm:cao}.
We first identify sufficient conditions to allow a ray start.
We define an \emph{implicit extreme point}, which is a generalization of 
an extreme point when the lineality space is nontrivial, and show that
there exists an implicit extreme point satisfying these sufficient conditions.
A computational method for finding such an implicit extreme point is described
in Section~\ref{sec:computing}. Our conditions generalize those required
for existing pivotal methods~\cite{cao96,dirkse95,lemke65} for \LCP{}, \MCP{},
and \AVI{}.

\PATHAVI{} can process $L$-matrices with respect to the recession cone of the
feasible set of the \AVI{}. To this end, we show that a 1-manifold
(the path $G_C^{-1}(0)$) generated by \PATHAVI{} with a ray start at an implicit
extreme point corresponds to a 1-manifold generated by the same pivotal
method with a ray start at an extreme point in the reduced space. The
reduced space is formed by projecting out the lineality space. 
This one-to-one correspondence is derived from
the structural correspondence of the faces and the normal cones between the
original space and the reduced one.
Then by applying the existing processability result
to the 1-manifold in the reduced space, we obtain the desired result.

\subsection{Sufficient conditions for a ray start and 
  processability of \PATHAVI{}}\label{sec:process}

We first identify sufficient conditions to perform a ray start at a point.

\begin{proposition}
Let an $\AVI(C,q,M)$ be given. If the following conditions are satisfied
at a point $\bar{z}$, then we can perform a ray start at $\bar{z}$.
\begin{itemize}
\item[$\bullet$] $M\bar{z} + q \in \aff(N_C(\bar{z}))$.
\item[$\bullet$] Every point in the interior of the $(n+1)$-cell 
  $((\bar{z}+\lin C)+N_C(\bar{z})) \times \mathbb{R}_+$ is regular.
\item[$\bullet$] There exists a complementary basis at $\bar{z}$
  such that $\aff(N_C(\bar{z}))$ is spanned by columns of the basic variables in
  $(\lambda,w,v)$.
\end{itemize}
\label{prop:sufficient}
\end{proposition}
\begin{proof}
Pick a vector $r \in \ri(N_C(\bar{z}))$. Let $(z,\lambda,w,v,s)$ be the
complementary basic solution to \eqref{eq:comp-eq} and \eqref{eq:comp-perp} 
corresponding to the given complementary
basis. Note that $z=\bar{z}$, thus $s$ is feasible. Therefore only basic
variables in $(\lambda,w,v)$ might be infeasible.
The first and third conditions say that we have
$Mz+q - A^\tr \lambda - w + v = 0$. By the third condition, for each
$t \ge 0$ we have a unique $(\lambda(t),w(t),v(t))$ satisfying
$Mz+q - A^\tr \lambda(t) - w(t) + v(t)  - tr = 0$. 
As $r \in \ri(N_C(\bar{z}))$, there exists $t^0 \ge 0$ such that
for all $t \ge t^0$ we have $Mz+q - A^\tr \lambda(t) - w(t) + v(t)  - tr = 0$
and $(\lambda(t),w(t),v(t))$ are feasible variables. 
Then for all $t \ge t^0$ $(x(t),t)$ with 
$x(t):=\bar{z} - A^\tr \lambda(t) - w(t) + v(t)$
lies in the cell $((\bar{z} + \lin C) + N_C(\bar{z})) \times \mathbb{R}_+$
with $\pi_C(x(t))=\bar{z}$ and $G_C(x(t),t)=0$. 
By the second condition, the ray $(x(t),t)$ is generated
at a regular point. By pivoting the $t$ variable into the complementary
basis, we see that we can perform a ray start at $\bar{z}$.
\end{proof}

Note that the sufficient conditions are satisfied at an extreme point. If
$z$ is an extreme point, then $\aff(N_C(z))\equiv\mathbb{R}^n$ thus the first
condition is trivially satisfied. Each extreme point has a corresponding
basic feasible solution (BFS) to $Ax-b=s$ \cite[Section 3.4]{murty76}, and
with that BFS we can construct a complementary
basis satisfying the third condition as shown in Proposition~\ref{prop:comp}
later in this paper. The second condition is also satisfied as proved in
Proposition~\ref{prop:regular}. 
As the existing pivotal methods \cite{cao96,dirkse95,lemke65} for \LCP{},
\MCP{}, and \AVI{} perform a ray start at an extreme point, we see that the
sufficient conditions generalize the existing result.

We now define an implicit extreme point, which is a generalization
of an extreme point when the lineality space is nontrivial.


\begin{definition}
Let $C$ be a convex set in $\mathbb{R}^n$. A point $z \in C$ is called
\textit{an implicit extreme point} of $C$ if
$z=\lambda z^1 + (1-\lambda)z^2$ for any $z^1,z^2 \in C$ and
$\lambda \in (0,1)$ implies that $z-z^1 \in \lin C$ and $z-z^2 \in \lin C$.
\label{def:implicit}
\end{definition}

Note that if the lineality space of $C$ is trivial,
that is, $\lin C=\{0\}$, then the definition of an implicit extreme point
coincides with definition of an extreme point.

In the following four propositions, we study some characteristics of
implicit extreme points. These characteristics are a generalization of
those of extreme points. They are used as a tool for showing the
existence of an implicit extreme point satisfying the sufficient conditions
and for structural analysis later in this section.
We start with faces consisting of only implicit extreme points. 
This generalizes 0-dimensional faces that are equivalent to extreme points.
As the proof is elementary, we omit it.

\begin{proposition}
Let $C$ be a nonempty convex set in $\mathbb{R}^n$ and $\ell = \dim(\lin C)$.
Then every point in an $\ell$-dimensional face of $C$ is an implicit extreme
point of $C$. Also for each implicit extreme point $z$ of $C$ we have
$F=z + \lin C$ is an $\ell$-dimensional face of $C$.
\label{prop:implicit-face}
\end{proposition}

We prove next that the affine hull of the normal cone to $C$ at an implicit
extreme point is the orthogonal complement of the lineality space of $C$.
This generalizes the fact that the normal cone to $C$ at an extreme point
is full-dimensional.

\begin{proposition}
A point $z$ is an implicit extreme point of a nonempty
polyhedral convex set $C$ in $\mathbb{R}^n$ if and only if
$z \in C$ and $\aff(N_C(z))=(\lin C)^\perp$.
\label{prop:normal-cone-perp}
\end{proposition}
\begin{proof}
(only-if) Suppose that $z$ is an implicit extreme point of $C$.
Using Proposition~\ref{prop:implicit-face}, $F=z+\lin C$ is a face
of $C$. We then have $\pspace F = \lin C$.
By \cite[Proposition 2.1]{robinson92}, $\pspace F=(\aff N_F)^\perp$,
where $N_F$ represents the normal cone having the same value for all
$\hat{z} \in \ri F$, i.e., $N_C(\hat{z}^1)=N_F=N_C(\hat{z}^2)$ for all
$\hat{z}^1,\hat{z}^2 \in \ri F$. As $z \in \ri F$, it follows that
$\aff(N_C(z))=(\lin C)^\perp$.

(if) Suppose that $z \in C$ and $\aff(N_C(z))=(\lin C)^\perp$.
Pick a face $F$ of $C$ such that $z \in \ri F$. Such a face exists
by \cite[Theorem 18.2]{rockafellar70}. Then $N_C(z)=N_F$, where
$N_F$ is the normal cone having constant value on $\ri F$.
As $\pspace F = (\aff N_F)^\perp$, we then have $\pspace F = \lin C$.
Thus $F=z+\lin C$. By Proposition~\ref{prop:implicit-face},
$z$ is an implicit extreme point of $C$.
\end{proof}

Next we show that the second condition in 
Proposition~\ref{prop:sufficient} is satisfied at an implicit extreme point.
Note that in the proposition below we show $\dim(M_C(\sigma))=n$, which
implies that $\dim(G_C(\sigma\times\mathbb{R}_+))=n$.

\begin{proposition}
Let $z$ be an implicit extreme point of a nonempty polyhedral convex
set $C$ in $\mathbb{R}^n$ and $\sigma$ be the cell 
$((z+\lin C) + N_C(z))$ in the normal manifold of $C$. 
Then for an $\AVI(C,q,M)$ with $M$ invertible on the lineality space of
$C$, we have $\dim(M_C(\sigma))=n$. 
\label{prop:regular}
\end{proposition}
\begin{proof}
By \cite[Proposition 2.5]{robinson92}, $M_C$ coincides with some affine
transformation $A_\sigma$ on $\sigma$. In the basis $Z=(Q \ \ \bar{Q})$, 
we can represent the matrix $A_{\sigma}(\cdot) - A_{\sigma}(z)$ as follows:
\begin{equation*}
  \begin{bmatrix}[1.5]
    Q^\tr M Q &\quad 0\\
    \bar{Q}^\tr M Q &\quad I
  \end{bmatrix}.
\end{equation*}
As $Q^\tr M Q$ is invertible, the matrix $A_{\sigma}(\cdot) - A_{\sigma}(z)$
is invertible. As $\sigma$ is $n$-dimensional, the result follows.
\end{proof}

Finally, the last characteristic presents that on each $\ell$-dimensional 
face $F$ with $\ell=\dim(\lin C)$
(hence consisting of only implicit extreme points by 
Proposition~\ref{prop:implicit-face}) there exists an implicit extreme
point $z \in F$ such that $Mz+q \in \aff(N_C(z))$. This generalizes the fact
that at each extreme point $\bar{z}$ we have
$M\bar{z}+q \in \aff(N_C(\bar{z}))\equiv\mathbb{R}^n$.

\begin{proposition}
Let an $\AVI(C,q,M)$ problem be given and $z \in C$ be an implicit
extreme point of $C$. Assume that $M$ is invertible on the lineality space of
$C$. Then there exists $\hat{z} \in z + \lin C$ such that
$M\hat{z} + q \in \aff(N_C(\hat{z}))$.
\label{prop:existence}
\end{proposition}
\begin{proof}
For any implicit extreme point $\hat{z}$ of $C$, 
$M\hat{z}+q \in \aff(N_C(\hat{z}))$ if and only if $\pi_{\lin C}(M\hat{z}+q)=0$
by Proposition~\ref{prop:normal-cone-perp}.
By the assumption, $M_{QQ}$ is invertible.
Set
\[
\hat{z} = z + Qy \quad \text{where} \quad
y = -M_{QQ}^{-1}(Q^\tr q + M_{Q\bar{Q}}\bar{Q}^\tr z) - Q^\tr z.
\]
Then $\hat{z} \in z + \lin C$ thus $\hat{z}$ is an implicit extreme point
of $C$ by Proposition~\ref{prop:implicit-face}, and
\begin{equation*}
  \begin{aligned}
    Q^\tr (M\hat{z}+q) &= Q^\tr \left(M
      \begin{bmatrix}Q & \bar{Q} \end{bmatrix}
      \begin{bmatrix}Q^\tr \\ \bar{Q}^\tr \end{bmatrix}\hat{z} + q
    \right),\\
    &= M_{QQ}(Q^\tr \hat{z}) + M_{Q\bar{Q}}(\bar{Q}^\tr \hat{z}) + Q^\tr q,\\
    &= M_{QQ}(Q^\tr z + y)+M_{Q\bar{Q}}(\bar{Q}^\tr z)+Q^\tr q,\\
    &= 0.
  \end{aligned}
\end{equation*}
It follows that $\pi_{\lin C}(M\hat{z}+q)=0$.
\end{proof}

By Propositions~\ref{prop:regular} and~\ref{prop:existence},
there exists an implicit extreme point satisfying the first two sufficient
conditions for a ray start.
We postpone checking the third condition to Section~\ref{sec:computing}
as it requires a constructive proof. For the rest of this section, we assume
that we have an implicit extreme point satisfying the sufficient conditions.

We now turn our attention to the processability of \PATHAVI{}. Assume that
we perform a ray start at an implicit extreme point and generate a 1-manifold
in the original space $\mathbb{R}^n$. Our basic idea of deriving
processability is that for this 1-manifold there corresponds to a 1-manifold
generated by the same pivotal method with a ray start at an extreme point 
defined in the reduced space having possibly smaller dimension. We can then use the
existing processability result \cite[Theorem 4.4]{cao96}. To establish the
correspondence, we prove that there is a one-to-one correspondence between
the faces,  the normal cones, and the full-dimensional cells of the original
space and reduced space as the following proposition shows. 

\begin{proposition}
Let $C$ be a nonempty polyhedral convex set in $\mathbb{R}^n$ and
$\tilde{C}$ be the set $\tilde{C}=\bar{Q}^{\tr}C=\{\tilde{z} \,|\,
\tilde{z}=\bar{Q}^{\tr}z \, \text{ for some }\,z \in C\}$ defined in
$\mathbb{R}^{n-\ell}$ where $\ell=\dim(\lin C)$. Then the followings
hold.
\begin{enumerate}[label=(\alph*)]
\item $z$ is an implicit extreme point of $C$ if and only if
  $\tilde{z}=\bar{Q}^{\tr}z$ is an extreme point of $\tilde{C}$.
\item $F$ is a face of $C$ if and only if
  $\tilde{F}=\bar{Q}^{\tr}F$ is a face of $\tilde{C}$.
\item $v \in N_C(z)$ if and only if $v = \bar{Q} \tilde{v}$ for some
  $\tilde{v} \in N_{\tilde{C}}(\tilde{z})$ where
  $\tilde{z}=\bar{Q}^{\tr}z$.
\item $\sigma$ is an $n$-cell of the normal manifold $\mathcal{N}_C$ of $C$
  if and only if $\tilde{\sigma}=\bar{Q}^\tr \sigma$ is an $(n-\ell)$-cell
  of the normal manifold $\mathcal{N}_{\tilde{C}}$.
\end{enumerate}
\label{prop:twokey}
\end{proposition}
\begin{proof}
We prove in sequence. (a) (only-if) Let $z$ be an implicit extreme point
of $C$. Set $\tilde{z}=\bar{Q}^\tr z$. We prove by contradiction.
Suppose that $\exists \tilde{z}^1,\tilde{z}^2 \in \tilde{C}$ and
$\lambda \in (0,1)$ such that $\tilde{z}=\lambda \tilde{z}^1+(1-\lambda)
\tilde{z}^2$ with $\tilde{z} \neq \tilde{z}^i$ for $i=1,2$. By definition
of $\tilde{C}$, we have $z^1,z^2 \in C$ such that
$\tilde{z}^i=\bar{Q}^\tr z^i$ for $i=1,2$.
As $C=\lin C \oplus ((\lin C)^\perp \cap C)$ \cite[page 65]{rockafellar70}
and $\bar{Q}^\tr z = \bar{Q}^\tr (\lambda z^1 + (1-\lambda)z^2)$,
there exists $a \in \lin C$ such that
$z = \lambda (a+z^1) + (1-\lambda)(a+z^2)$.
As $\bar{Q}^\tr (z- (a+z^i)) = \tilde{z} - \tilde{z}^i \neq 0$,
we have $z-(a+z^i) \notin \lin C$ for $i=1,2$, which contradicts our
assumption that $z$ is an implicit extreme point of $C$.

(if) Using similar proof technique, we can show that for an extreme point
$\tilde{z} \in \tilde{C}$ $z$ is an implicit extreme point of $C$
when $\tilde{z}=\bar{Q}^\tr z$.

(b) (only-if) Let $F$ be a face of $C$. Set $\tilde{F}=\bar{Q}^\tr F$.
Clearly, $\tilde{F}$ is a convex subset of $\tilde{C}$. Let
$\tilde{z}^1,\tilde{z}^2 \in \tilde{C}$ and $\lambda \in (0,1)$
satisfying $\lambda\tilde{z}^1 + (1-\lambda)\tilde{z}^2 \in \tilde{F}$.
From $C=\lin C \oplus ((\lin C)^\perp \cap C)$, we have
$\bar{Q} \tilde{z}^i \in C$ for $i=1,2$. Then
$\bar{Q} (\lambda\tilde{z}^1+(1-\lambda)\tilde{z}^2) \in F$ so that
$\bar{Q} \tilde{z}^1 \in F$ and $\bar{Q} \tilde{z}^2 \in F$. This shows
that $\tilde{z}^i \in \tilde{F}$ for $i=1,2$.

(if) Let $\tilde{F}=\bar{Q}^\tr F$ be a face of $\tilde{C}$. By the 
definition of
$\tilde{F}$, $F$ is a convex subset of $C$. Let $z^1,z^2 \in C$ and
$\lambda \in (0,1)$ such that $\lambda z^1 + (1-\lambda)z^2 \in F$. We have
$\bar{Q}^\tr z^i \in \tilde{C}$ for $i=1,2$ and
$\bar{Q}^\tr (\lambda z^1+(1-\lambda)z^2) \in \tilde{F}$. Thus
$\bar{Q}^\tr z^i \in \tilde{F}$, hence $z^i \in F + \lin C$ for $i=1,2$.
Therefore $z^i \in F$ for $i=1,2$.

(c) For a vector $v \in \mathbb{R}^n$, we represent components of $v$ in
$\lin C$ and $(\lin C)^\perp$ in the basis 
$\begin{bmatrix}Q \quad \bar{Q}\end{bmatrix}$
by $v_Q$ and $v_{\bar{Q}}$, respectively, so that
$v = Qv_Q + \bar{Q} v_{\bar{Q}}$. If either $z \notin C$ or
$\tilde{z} \notin \tilde{C}$, then we have
nothing to prove. Therefore we assume that $z \in C$ and
$\tilde{z} \in \tilde{C}$ in the proof. (only-if) Let $v \in N_C(z)$.
By the definition of the normal cone, for each $a \in \lin C$ we have
$\langle v, (z+a) - z\rangle \le 0$ and
$\langle v, (z-a) - z \rangle \le 0$. Thus $\langle v,a \rangle = 0$ for all
$a \in \lin C$. Thus $N_C(z) \subset (\lin C)^\perp$ so that $v_Q=0$ and
$v=\bar{Q} v_{\bar{Q}}$. We then have
\begin{equation*}
  \begin{aligned}
    0 &\ge \langle v, y - z \rangle, \quad \forall y \in C\\
    &=\langle \bar{Q} v_{\bar{Q}},Qy_Q+\bar{Q} y_{\bar{Q}} -
    (Qz_Q+\bar{Q} z_{\bar{Q}})\rangle\\
    &=\langle \bar{Q} v_{\bar{Q}},\bar{Q} (y_{\bar{Q}} - z_{\bar{Q}})\rangle\\
    &=\langle v_{\bar{Q}},y_{\bar{Q}} - z_{\bar{Q}} \rangle
  \end{aligned}
\end{equation*}
By setting $\tilde{v}=v_{\bar{Q}}$, $v=\bar{Q} \tilde{v}$ and
$\tilde{v} \in N_{\tilde{C}}(\tilde{z})$.

(if) Let $\tilde{v} \in N_{\tilde{C}}(\tilde{z})$ and set $v=\bar{Q} \tilde{v}$.
We have $\tilde{z} = \bar{Q}^\tr z$ if and only if
$z \in \lin C + \bar{Q} \tilde{z}$. Let $z \in \lin C + \bar{Q} \tilde{z}$.
Then
\[
\langle v, y - z \rangle = \langle \tilde{v},y_{\bar{Q}}-z_{\bar{Q}}\rangle
\le 0, \quad y \in C
\]
The result follows.

(d) The result follows from (b), (c), and the definition of the full-dimensional
cells of the normal manifold.
\end{proof}

A similar result holds for the 1-manifold $G_C^{-1}(0)$.

\begin{proposition}
Let an $\AVI(C,q,M)$ problem be given. Suppose that the matrix $M$ is
invertible on the lineality space of $C$, and
$G_C(x^*,t^*) = 0$ where $r \in N_C(\pi_C(x^0))$ for some $x^0 \in \mathbb{R}^n$.
Then the \PL{} function
$\tilde{G}_{\tilde{C}}(\tilde{x},t):=
\tilde{M}\pi_{\tilde{C}}(\tilde{x})+\tilde{q}+\tilde{x}
-\pi_{\tilde{C}}(\tilde{x})-t\tilde{r}$ has value zero at $(\tilde{x}^*,t^*)$,
where
\begin{equation*}
  \begin{aligned}
    \tilde{x}^* &= \bar{Q}^\tr x^*\\
    Z &= \begin{bmatrix} Q && \bar{Q} \end{bmatrix},\\
    \tilde{M} &= (Z^\tr M Z / M_{QQ})=M_{\bar{Q}\bar{Q}} -
    M_{\bar{Q}Q}M_{QQ}^{-1}M_{Q\bar{Q}},\\
    \tilde{C} &= \bar{Q}^\tr C, \, \tilde{x}^0 = \bar{Q}^\tr x^0, \,
    \tilde{q} = (\bar{Q}^\tr - M_{\bar{Q}Q}M_{QQ}^{-1}Q^\tr)q,\\
    \tilde{r} &= \bar{Q}^\tr r \in N_{\tilde{C}}(\pi_{\tilde{C}}(\tilde{x}^0)).
  \end{aligned}
\end{equation*}
Conversely, if $\tilde{G}_{\tilde{C}}(\tilde{x}^*,t^*)=0$ then
$G_C(x^*,t^*)=0$ with $x^*=\bar{Q}\tilde{x}^* + Qy^*$ and
\[
y^* = -M_{QQ}^{-1}\left(M_{Q\bar{Q}}\pi_{\tilde{C}}(\tilde{x}^*) + Q^\tr q\right).
\]
\label{prop:gpath}
\end{proposition}
\begin{proof}
Let $(x^*,t^*)$ satisfying $G_C(x^*,t^*)=0$ with $r \in N_C(\pi_C(x^0))$ for
some $x^0$ be given. Then
\begin{equation*}
  \begin{aligned}
    &M\pi_C(x^*) + q + x^* - \pi_C(x^*) - t^*r = 0,\\
    (\Rightarrow)\, &
    \begin{bmatrix}
      Q^\tr \\ \bar{Q}^\tr
    \end{bmatrix}
    M
    \begin{bmatrix}
      Q & \bar{Q}
    \end{bmatrix}
    \begin{bmatrix}
      Q^\tr \\ \bar{Q}^\tr
    \end{bmatrix}
    \pi_C(x^*) +
    \begin{bmatrix}
      Q^\tr \\ \bar{Q}^\tr
    \end{bmatrix}
    (q + x^* - \pi_C(x^*) - t^*r) = 0,\\
    (\Rightarrow)\, &
    \begin{bmatrix}
      M_{QQ} & M_{Q\bar{Q}}\\M_{\bar{Q}Q}&M_{\bar{Q}\bar{Q}}
    \end{bmatrix}
    \begin{bmatrix}
      Q^\tr \pi_C(x^*) \\ \bar{Q}^\tr \pi_C(x^*)
    \end{bmatrix}
    +
    \begin{bmatrix}
      Q^\tr q \\ \bar{Q}^\tr (q + x^* - \pi_C(x^*) - t^*r)
    \end{bmatrix}
    = 0,\\
    (\Rightarrow)\, &
    \tilde{M}\bar{Q}^\tr \pi_C(x^*) + \tilde{q} +
    \bar{Q}^\tr (x^* - \pi_C(x^*)) - t^*\tilde{r} = 0,\\
    &\quad \text{ using }
    Q^\tr \pi_C(x^*) =
    -M_{QQ}^{-1}(M_{Q\bar{Q}}\bar{Q}^\tr \pi_C(x^*)+Q^\tr q),\\
    (\Rightarrow)\, &
    \tilde{M}\pi_{\tilde{C}}(\tilde{x}^*) + \tilde{q} + \tilde{x}^*
    - \pi_{\tilde{C}}(\tilde{x}^*) - t^* \tilde{r} = 0.
  \end{aligned}
\end{equation*}
The second $(\Rightarrow)$ holds because $N_C(\pi_C(x^0))\subset(\lin C)^\perp$ .
The last $(\Rightarrow)$ holds because
$\bar{Q}^\tr \pi_C(x)=\pi_{\bar{Q}^\tr C}(\bar{Q}^\tr x)$ by
\cite[Lemma 2.1]{cao95}. Also 
$\tilde{r} \in N_{\tilde{C}}(\pi_{\tilde{C}}(\tilde{x}^0))$
by Proposition~\ref{prop:twokey}(d). 

Conversely, let $\tilde{G}_{\tilde{C}}(\tilde{x}^*,t^*)=0$.
Set $x^*=\bar{Q}\tilde{x}^* + Qy^*$ with $y^*$ as specified in the proposition.
Then
\begin{equation*}
  \begin{aligned}
    \pi_C(x^*) &= \pi_{C \cap (\lin C)^\perp}(x^*) + \pi_{\lin C}(x^*)\\
    &= \bar{Q}\pi_{\tilde{C}}(\tilde{x}^*) + Qy^*
  \end{aligned}
\end{equation*}
Therefore $Q^\tr \pi_C(x^*)= y^*$. By the definition of $y^*$, the converse
directions also hold.
\end{proof}

Note that the $\AVI(\tilde{C},\tilde{q},\tilde{M})$ with
$\tilde{C},\tilde{q}$, and $\tilde{M}$ as in Proposition~\ref{prop:gpath}
is the same exact problem obtained by applying the stage 1 reduction
\cite[page 49]{cao96} to the $\AVI(C,q,M)$.
Also $\tilde{G}_{\tilde{C}}$ is the \PL{} function defined on the 
$(n-\dim(\lin C)+1)$-manifold $\mathcal{M}_{\tilde{C}}$ of $\tilde{C}$ to find
a zero of the normal map associated with the
$\AVI(\tilde{C},\tilde{q},\tilde{M})$.

An implication of Proposition~\ref{prop:gpath} is that if
$G_C(x+\theta\Delta x,t+\theta\Delta t)=0$ and
$(x+\theta\Delta x,t+\theta\Delta t) \in \sigma \times \mathbb{R}_+$
for all $\theta \in [0,\nu]$ for some $\nu > 0$, possibly $\nu = \infty$,
and $\sigma \times \mathbb{R}_+$ is an $(n+1)$-cell of $\mathcal{M}_C$,
then we have 
$\tilde{G}_{\tilde{C}}(\tilde{x}+\theta\Delta\tilde{x},t+\theta\Delta t)=0$
with $(\tilde{x}+\theta\Delta\tilde{x},t+\theta\Delta t)
\in \tilde{\sigma}\times\mathbb{R}_+$ for all
$\theta \in [0,\nu]$, where $\tilde{\sigma}=\bar{Q}^\tr \sigma$ and
$\Delta\tilde{x}=\bar{Q}^\tr \Delta x$. The converse also holds
by setting $\Delta x = \bar{Q}\Delta\tilde{x} + Q\Delta y$ with
$\Delta y=-M_{QQ}^{-1}M_{Q\bar{Q}}H_{\tilde{\sigma}}\Delta \tilde{x}$, where
$H_{\tilde{\sigma}}$ is a matrix representing the projection operator
$\pi_{\tilde{C}}(\cdot)$ on elements of $\tilde{\sigma}$.
Therefore the projection of each piece of $G_C^{-1}(0)$ onto
$\mathcal{M}_{\tilde{C}}$ corresponds to each piece of
$\tilde{G}_{\tilde{C}}^{-1}(0)$ and vice versa. 
As a consequence, if $G^{-1}_C(0)$ contains
a ray, i.e., $\exists (\Delta x, \Delta t) \neq 0$ with $\nu = \infty$ on some
$(n+1)$-cell $\sigma \times \mathbb{R}_+$ of $\mathcal{M}_C$, and
the corresponding value $(\Delta \tilde{x},\Delta t)$ is not zero, then the
corresponding piece of $\tilde{G}^{-1}_{\tilde{C}}(0)$ is also a ray.
The following proposition shows that whenever there is a ray in $G_C^{-1}(0)$
with $\Delta x \neq 0$, then we have $\Delta \tilde{x} \neq 0$ so that
the corresponding piece of $\tilde{G}_{\tilde{C}}^{-1}(0)$ is also a ray.
Note that the converse automatically holds as $\Delta x \neq 0$ for each
$\Delta \tilde{x} \neq 0$.

\begin{proposition}
For an $\AVI(C,q,M)$, suppose that \PATHAVI{} generates $G^{-1}_C(0)$
with a ray start at an implicit extreme point. 
For each ray in $G_C^{-1}(0)$ in the direction of $(\Delta x,\Delta t) \neq 0$,
if $\Delta x \neq 0$ 
then $\Delta \tilde{x} := \bar{Q}^\tr \Delta x$ is a ray in
$\tilde{G}_{\tilde{C}}^{-1}(0)$ that is nonzero
under the assumption that either $0$ is a regular value or we do
lexicographic pivoting.
\label{prop:gpath-nolin}
\end{proposition}
\begin{proof}
Let $z$ be an implicit extreme point at which \PATHAVI{} performs a ray start.
By construction, $\Delta \tilde{x}=0$ if and only if $\Delta x \in \lin C$.
For the starting ray we have $\Delta x \in N_C(z)$ so that
$\Delta x \notin \lin C$ by Proposition~\ref{prop:normal-cone-perp}.
Thus $\Delta \tilde{x} \neq 0$. 

We now assume that there is a ray in $G^{-1}_C(0)$ other than the starting ray.
Suppose that $\Delta x \in \lin C$. We prove by contradiction.
Let us assume that the ray is generated at the $(k+1)$th iteration of
complementary pivoting, and it starts from $x^{k+1}$. We know that
$x^{k+1} \in \sigma^{k+1} \times \mathbb{R}_+$ and
$x^{k+1} \in \sigma^k \times \mathbb{R}_+$, where $\sigma^k \times \mathbb{R}_+$
is the $(n+1)$-cell of $\mathcal{M}_C$ \PATHAVI{} passes through at the
$k$th complementary pivoting iteration.
As $\lin C \subset \lin \sigma$ for each $(n+1)$-cell
$\sigma \times \mathbb{R}_+$ of $\mathcal{M}_C$, 
$x^{k+1} + \theta \Delta x \in \sigma^k$ for all $\theta \ge 0$.
This contradicts the fact that $G^{-1}_C(0)$ is a 1-manifold
neat in $\mathcal{M}_C$ \cite[Theorem 9.1 or Lemma 15.5]{eaves76}, that is,
$G^{-1}(0) \cap (\sigma^k \times \mathbb{R}_+)$ must be expressed as an 
intersection of $\sigma^k \times \mathbb{R}_+$ with a line.
Therefore $\Delta x \notin \lin C$. The result follows.
\end{proof}
From Lemma~\ref{lem:aux_var_cst}, if $M$ is semimonotone with respect to $\rec C$
and invertible on $\lin C$,
we have $\Delta t = 0$ whenever \PATHAVI{} generates a ray
in the direction of $(\Delta x, \Delta t)$. These classes include
the $L$-matrix class and the new matrix classes defined in
 Section~\ref{sec:more-processability}. Therefore whenever
\PATHAVI{} generates a ray in $G_C^{-1}(0)$ for those classes of matrices
the corresponding piece in $\tilde{G}^{-1}_{\tilde{C}}(0)$ is also a ray by
Proposition~\ref{prop:gpath-nolin}.

With Propositions~\ref{prop:twokey}--\ref{prop:gpath-nolin},
we finally prove that \PATHAVI{} can process $L$-matrices as defined in
\cite[Theorem 4.4]{cao96}. In contrast to \cite{cao96}, the following
proof does not apply any reduction to achieve the result. 

\begin{theorem}
Suppose that $C$ is a polyhedral convex set, and $M$ is an $L$-matrix with
respect to $\rec C$ which is invertible on the lineality space of $C$. Then
exactly one of the following occurs:
\begin{itemize}
\item[$\bullet$] \PATHAVI{} solves the $\AVI(C,q,M)$.
\item[$\bullet$] The following system has no solution
  \begin{equation*}
    Mz + q \in (\rec C)^D.
  \end{equation*}
\end{itemize}
\label{thm:class}
\end{theorem}
\begin{proof}
By Propositions~\ref{prop:twokey}--\ref{prop:gpath},
for a 1-manifold $G_C^{-1}(0)$ generated by
\PATHAVI{} there corresponds to a 1-manifold $\tilde{G}_{\tilde{C}}^{-1}(0)$
in the reduced space generated by the same pivotal method with a ray start
at an extreme point of $\tilde{C}$ with $\tilde{M}$ an $L$-matrix with
respect to $\rec \tilde{C}$. If there is a secondary ray in $G_C^{-1}(0)$,
then so is in $\tilde{G}_{\tilde{C}}^{-1}(0)$ by
Proposition~\ref{prop:gpath-nolin}.
Therefore there exists directions
$(\Delta\tilde{x},\Delta\tilde{z},\Delta\tilde{\lambda},\Delta\tilde{s},
\Delta t)$ in the reduced space satisfying
\begin{equation}
  \begin{aligned}
   \Delta \tilde{x} - \Delta \tilde{z} &= -\tilde{M}\Delta \tilde{z} + r\Delta t,\\
    A_{\mathcal{A}\bullet}\Delta\tilde{z} &= 0,\\
    A_{\bar{\mathcal{A}}\bullet}\Delta\tilde{z}-\Delta\tilde{s}_{\bar{\mathcal{A}}}
    &= 0,\\
    \Delta\tilde{x}-\Delta\tilde{z} &= 
    -A^\tr_{\mathcal{A}\bullet}\Delta\tilde{\lambda}_{\mathcal{A}}
   \end{aligned}\label{eq:sys_Deltas}
\end{equation}
where we have included bound constraints in the matrix $A$ for clarity, and
$\mathcal{A}$ and $\bar{\mathcal{A}}$ denote the active and inactive sets,
respectively.
We then apply Theorem \ref{thm:cao} to \eqref{eq:sys_Deltas}
to get the desired result.
\end{proof}

\subsection{Additional processability results}
\label{sec:more-processability}

Let us now extend the classes of \AVI{}s that \PATHAVI{} is able to process.
The results in 
Lemmas~\ref{lem:pathAVIext_Lmat}--\ref{lem:pathAVIext_copos} 
consider the structure of the whole \AVI{}, not only $M$ and $C$.
As stated previously, a 1-manifold generated in $\mathcal{M}_C$ corresponds to another one in $\mathcal{M}_{\reds{C}}$.
Hence, in the following we denote by \AVI{}$(\reds{C},\reds{q}, \reds{M})$ the \AVI{} corresponding to \AVI{}$(C, q, M)$
with the lineality space projected out. If $M$ is invertible on $\lin C$,
the results can then be applied to original \AVI{} by noting that the projections of the directions of
the rays on $G_C^{-1}(0)$ are solution to the system of equations~\eqref{eq:sys_Deltas} in the reduced space.

In Section~\ref{subsec:FCP}, we present a friction contact problem
where Theorem~\ref{thm:class} cannot be applied but the following lemma can.

\begin{lemma}\label{lem:pathAVIext_Lmat}
 Consider an \AVI{}$(\reds{C}, \reds{q}, \reds{M})$ with $\lin \reds{C} = \{0\}$. Suppose that $\reds{M}$ is semimonotone with respect to $\rec \reds{C}$
and that for any solution $z\neq0$ of the problem
\begin{equation}
 z\in \rec \reds{C},\quad \reds{M} z\in\DUAL{(\rec \reds{C})},\quad z^\tr \reds{M}z = 0,\label{eq:fc_vi:Lm_b2bis}
\end{equation}
it holds that
\begin{equation}
 z^\tr(\reds{M}z'+ \reds{q})\geq0, \quad\forall z'\in \reds{C}. \label{eq:fc_vi:pathAVIext_cond}
\end{equation}
Then \PATHAVI{} solves the \AVI{}$(\reds{C}, \reds{q}, \reds{M})$.
\end{lemma}
\begin{proof}
  The pivotal method used in \PATHAVI{} fails if an unbounded ray is generated at some iterate $(\MYK{x}, \MYK{t})$, $k>0$.
  Now suppose that the method generates an unbounded ray.
  From Lemma~\ref{lem:aux_var_cst} we know that $\Delta t =0$, and $\Delta z \neq 0$ is a solution to~\eqref{eq:fc_vi:Lm_b2bis}.
  This means that for any point $\MYKp{x}$ on the ray, we have $\reds{G}_\reds{C}(\MYKp{x}, \MYK{t}) = 0$, implying that
\begin{equation}
 \langle \Delta z, \reds{G}_{\reds{C}}(\MYKp{x}, \MYK{t})\rangle = \langle \Delta z, \reds{M}\MYKp{z} + \reds{q}\rangle + \langle \Delta z, \MYKp{x} - \MYKp{z}\rangle + \langle \Delta z, -\MYK{t} r\rangle = 0.
\end{equation}
 The first term is non-negative by our assumption, as well as the second one by the normal cone definition.
 The third one is strictly positive since $-\MYK{t} r\in\intS \DUAL{(\rec \reds{C})}$. Hence we reached a contradiction.
\end{proof}
An additional property on $\tilde{M}$ allows easier checking of the
conditions \eqref{eq:fc_vi:pathAVIext_cond} of Lemma \ref{lem:pathAVIext_Lmat}.
\begin{corollary}
 If for any solution $z$ of~\eqref{eq:fc_vi:Lm_b2bis} we have $\langle z', \reds{M}^\tr z\rangle \geq 0$, for all $z'\in \reds{C}$, then the condition~\eqref{eq:fc_vi:pathAVIext_cond}
 reduces to $z^\tr \reds{q}\geq0$ whenever $z$ is a solution to \eqref{eq:fc_vi:Lm_b2bis}.
\end{corollary}

We introduce an additional problem class \PATHAVI{} can process.
\begin{lemma}\label{lem:pathAVIext_copos}
 Consider an \AVI{}$(\reds{C}, \reds{q}, \reds{M})$ with $\lin \reds{C} = \{0\}$. Suppose that $\reds{C}$ is a proper cone, $\reds{M}$ is copositive with respect to $\reds{C}$
 and that the following implication holds:
\begin{equation}
 z\in \rec \reds{C},\quad \reds{M} z\in\DUAL{(\rec \reds{C})},\quad z^\tr \reds{M}z = 0\quad\Rightarrow\quad z^\tr \reds{q} \geq 0.\label{eq:fc_vi:q_cond}
\end{equation}
Then the \AVI{}$(\reds{C}, \reds{q}, \reds{M})$ has a solution and \PATHAVI{} finds it.
\end{lemma}
\begin{proof}
 Recall from~\cite[Lemma~4.3]{cao96}, that a copositive matrix is also semimonotone.
 This implies that $\Delta t = 0$ and that $\Delta z \neq 0$ satisfies the left-hand side of~\eqref{eq:fc_vi:q_cond}.
Now let us suppose that at the current iterate $x_k$, there exists an unbounded ray.
Letting $z_{k+1} = z_k + \theta\Delta z$  and computing the inner product 
$\langle z_{k+1}, \reds{G}_{\reds{C}}(x_{k+1},t_k)\rangle$ yields
\begin{equation}
 0=\langle z_{k+1}, \reds{G}_{\reds{C}}(x_{k+1}, t_k) \rangle = 
  \langle z_{k+1}, \reds{M}z_{k+1}\rangle + \langle z_{k+1}, \reds{q}\rangle + \langle z_{k+1}, x_{k+1} - z_{k+1}\rangle + \langle z_{k+1}, -t_k r\rangle.
  \label{}
\end{equation}
Note that since $\reds{C}$ is pointed, $\langle z_{k+1}, x_{k+1} - z_{k+1}\rangle \geq 0$ by the definition of the normal cone.
The first term is quadratic in $\theta$ while the second and third are linear in $\theta$.
Therefore, if $\langle \Delta z, \reds{M}\Delta z\rangle > 0$, then
$\langle z_{k+1}, \reds{G}_{\reds{C}}(x_{k+1}, t_k) \rangle > 0$ for $\theta$ large enough
and we reach a contradiction. We are left with the case $\langle \Delta z, \reds{M}\Delta z\rangle = 0$:
\begin{gather}
 0=
 \langle z_{k}, \reds{q} \rangle - \langle z_{k}, t_k r\rangle + \langle z_{k+1}, x_{k+1} - z_{k+1}\rangle + \langle z_{k+1}, \reds{M}z_{k+1}\rangle + \theta(\langle \Delta z, \reds{q}\rangle + \langle \Delta z, -t_k r\rangle).  \label{eq:fc_vi:copos_theta_z}
\end{gather}
The sum multiplied by $\theta$ is positive since $-t_k r\in\intS \DUAL{(\rec \reds{C})}$.
Now the first two terms are constant and the third and fourth ones are nonnegative. Whence for $\theta$ large enough, $\langle z_{k+1}, \reds{G}_{\reds{C}}(x_{k+1}, t_k)\rangle$ is positive,
which concludes the proof.
\end{proof}
\begin{remark}
 Lemma~\ref{lem:pathAVIext_copos} was already known for the LCP case (that is $\reds{C}=\R^n_+$):
 the existence of a solution is given in~\cite[Theorem~3.8.6]{cottle2009linear}.
 Here we are able to provide a constructive proof for an $\AVI{}(\reds{C}, \reds{q}, \reds{M})$.
\end{remark}
Let us present an \AVI$(C, q, M)$ that satisfies the conditions of Lemma~\ref{lem:pathAVIext_copos} where $M$ is not an $\Lmat$.
Suppose that $C \subseteq \R^{n+1}_+$ is a polyhedral solid cone, $M= \begin{pmatrix} I_n&0\\\mathbf{1}_n^{\tr}&0 \end{pmatrix}$, with $\mathbf{1}_n$ the vector of ones of size $n$ and $q=(0_n, 1)^{\tr}$.
The solution set of the system $x\in C$, $Mx = 0$  and $x^{\tr}Mx = 0$ is $\{(0_n, \alpha)^{\tr}$, $\alpha \geq 0\}$.
Note that if $x = (0_n, \alpha)^{\tr}$, $\alpha > 0$, then $Ax = 0$ and $x^{\tr}Mx = 0$.
However, for any nonzero vector $x'=(x_1'^{\tr}, \alpha')^{\tr}$ in $ C$, $-M^{\tr}x' = (-I_n x_1'^{\tr} -\alpha'\mathbf{1}^{\tr}, 0)^{\tr}\not\in\DUAL{C}$.
Therefore, condition (b) of the $\Lmat$ fails to hold. On the other hand, we can readily check that $M$ is copositive with respect to $C$ and that
for any $x = (0_n, \alpha)^{\tr}$, $\alpha \geq 0$, $x^{\tr}q = \alpha \geq 0$. Whence Lemma~\ref{lem:pathAVIext_copos} can be used.

\section{Computing an implicit extreme point
  for a ray start}\label{sec:computing}

In this section, we describe how to compute an implicit extreme point
satisfying the sufficient conditions for a ray start and the complementary
basis associated with it so that we can start complementary pivoting at that
implicit extreme point.
An implicit extreme point is computed using a linear programming (LP) solver,
i.e., CPLEX or GUROBI, with possibly additional pivoting, and
its complementary basis is computed based on the basis information given by
the LP solver. The use of the existing LP solver, which has fast sparse
linear algebra engine and pivoting method, as well as the use of
sparse linear algebra engine for complementary pivoting enables \PATHAVI{}
to fully exploit the sparse representation of the given \AVI{}. 
This makes our method efficient for large-scale \AVI{} problems.
See for example Section \ref{subsec:exp-preserve}.

We start with an introduction to some terminology and notational conventions
for describing a basic solution of an LP problem. We follow notation used in
\cite{bixby92}.
Suppose that we run an LP solver over an LP problem: minimize $c^\tr z$
subject to $Az - b \in K$ and $l \le z \le u$.
Without loss of generality, we assume that we have eliminated all fixed
variables.
For each solution $z$ obtained from the LP solver, we have four index sets,
$B,N_l,N_u$, and $N_{fr}$, for variables and
two index sets, $\mathcal{A}$ and $\bar{\mathcal{A}}$, for constraints
described by $A$ and $b$. \footnote{These index sets can be obtained using
\texttt{CPXgetbase()} for CPLEX, for example.}
Table~\ref{tbl:basis} in the Appendix lists the properties of
the index sets and the solution $z$. In Table~\ref{tbl:basis},
if $l_B \le z_B \le u_B$, we say that $z$ is a basic feasible solution.
Otherwise, we say that $z$ is a basic solution. Note that we have
$|\mathcal{A}|=|B|$ in Table~\ref{tbl:basis} as the basis matrix
$\mathbf{B}$ is invertible.
Hence the submatrix $A_{\mathcal{A}B}$ of $\mathbf{B}$ is square and invertible.

We first describe how to compute an implicit extreme point of $C$,
based on which we compute another implicit extreme point if necessary
satisfying the sufficient conditions for a ray start. For a given $\AVI(C,q,M)$,
we formulate and solve the following LP problem using an LP solver:
\begin{equation*}
\begin{aligned}
&\text{minimize} && 0^\tr z\\
&\text{subject to} && Az - b \in K\\
&&& l \le z \le u
\end{aligned}
\tag{LP}
\label{eq:lp}
\end{equation*}

We put zero objective coefficients in the \eqref{eq:lp} so that
the \eqref{eq:lp} returns right away once it finds a basic feasible solution;
the (LP) returns once it solves the phase I of the simplex method.
If we have a good knowledge about where to start complementary pivoting, then
we could try to solve the (LP) with different objective coefficients.

Assuming that the \eqref{eq:lp} is feasible, a basic feasible solution
$z^0$ from the LP solver with the corresponding index sets is an extreme
point if $N_{fr}=\varnothing$.
When $N_{fr} \neq \varnothing$, $z^0$ might not be an implicit extreme point.
In this case, we move from $z^0$ to another implicit extreme point by doing
additional pivoting in a way that we make as many nonbasic free variables
as basic variables. Algorithm~\ref{alg:pivoting} in the Appendix describes
the pivoting procedure. After applying Algorithm~\ref{alg:pivoting}, for each
$j \in N_{fr}$ and $d^j=A_{\mathcal{A}B}^{-1}A_{\mathcal{A},j}$ if there exists
$k$ such that $d^j_k \neq 0$, then the basic variable corresponding to the
$k$th position in $B$ is a free variable. Otherwise, the variable
$z_j$ must have been pivoted in by Algorithm~\ref{alg:pivoting}.
Also note that Algorithm~\ref{alg:pivoting} doesn't change the properties
described in Table~\ref{tbl:basis}.
Using Algorithm~\ref{alg:pivoting}, we obtain the following result.

\begin{proposition}
Suppose that we have applied Algorithm~\ref{alg:pivoting}. Then the new point,
denoted by $\bar{z}^0$,
constructed from $z^0$ through Algorithm~\ref{alg:pivoting}
is an implicit extreme point of $C$. We have
$\dim(\lin C)=|N_{fr}|$ and the following set of vectors is a basis for
the lineality space of $C$:
\begin{equation*}
\bigcup_{j \in N_{fr}} \{v^j\},\quad
v_k^j = \left\{\begin{array}{ll}
(A_{\mathcal{A}B}^{-1}A_{\mathcal{A},j})_k & \text{if } k \in B,\\
0 & \text{if } k \in N_l \cup N_u,\\
0 & \text{if } k \in N_{fr}, k \neq j,\\
1 & \text{if } k = j.
\end{array}
\right.
\end{equation*}
\end{proposition}
\begin{proof}
Clearly, $\bar{z}^0 \in C$ as we do a ratio test to move the point.
We first show that $\lin C = |N_{fr}|$ and
$\{v^j\}_{j \in N_{fr}}$ is a basis for the lineality space of $C$.
For each $j \in N_{fr}$,
if $v^j_k \neq 0$ for $k \in B$, then we have $l_k = -\infty$ and
$u_k = \infty$ as discussed in the previous paragraph.
It follows that $\bar{z}^0 + \lambda v^j \in C$ for all
$\lambda \in \mathbb{R}$. By \cite[Theorem 8.3]{rockafellar70},
$v^j \in \rec C \cap (-\rec C)$. Thus $v^j \in \lin C$.
By construction of $v^j$, we see that $v^j$'s are linearly independent.
This implies that $\dim(\lin C) \ge |N_{fr}|$.
As $\dim(N_C(\bar{z}^0)) \ge |B|+|N_l|+|N_u|$
and $N_C(\bar{z}^0) \subset (\lin C)^\perp$ as shown in
Proposition~\ref{prop:twokey}(d), it follows that $\dim(\lin C) = |N_{fr}|$ and
$\{v^j\}_{j \in N_{fr}}$ is a basis for the lineality space of $C$.

We now prove that $\bar{z}^0$ is an implicit extreme point of $C$.
Suppose that $\bar{z}^0=\lambda z^1 + (1-\lambda)z^2$ for some $z^1,z^2 \in C$
and $\lambda \in (0,1)$. Define $d^k = \sum_{j \in N_{fr}} (-z^k_j v^j)$
and set $\tilde{z}^k=z^k + d^k$ for $k=1,2$. We then have
$\tilde{z}^k_j = 0$ for $j \in N_{fr}$ and $\tilde{z}^k \in C$ as
$d^k \in \lin C$ for $k=1,2$.
As $\bar{z}^0=\lambda z^1 + (1-\lambda)z^2$,
$\bar{z}^0=\lambda \tilde{z}^1 + (1-\lambda)\tilde{z}^2 -
(\lambda d^1 + (1-\lambda)d^2)$. We have
$\lambda d^1 + (1-\lambda)d^2
=\sum_{j \in N_{fr}}\left(-(\lambda z^1_j+(1-\lambda)z^2_j)v^j\right)$. As
$\bar{z}^0_{N_{fr}}=\tilde{z}^1_{N_{fr}}=\tilde{z}^2_{N_{fr}}=0, v^j_j=1$,
and $v^j_h=0$ for $h \in N_{fr}, h \neq j$,
we see that $\lambda d^1+(1-\lambda)d^2=0$. Therefore,
$\bar{z}^0=\lambda \tilde{z}^1+(1-\lambda)\tilde{z}^2$. It follows that
$\bar{z}^0=\tilde{z}^1=\tilde{z}^2$. Thus, $\bar{z}^0-z^k = d^k \in \lin C$
for $k=1,2$, which implies that $\bar{z}^0$ is an implicit extreme point of
$C$.
\end{proof}

With the implicit extreme point $\bar{z}^0$ of $C$ and the index sets
$(B,N_l,N_u,N_{fr},\mathcal{A},\bar{\mathcal{A}})$ associated with it,
we finally construct an initial complementary basis and compute an implicit
extreme point satisfying the sufficient conditions for a ray start using that
complementary basis. To prove the invertibility of our initial complementary
basis, we first need to introduce the following technical result
derived from \cite[Lemma 3.6]{liu95}.

\begin{corollary}
Suppose that we have index sets
$(B,N_l,N_u,N_{fr},\mathcal{A},\bar{\mathcal{A}})$ associated with
an $\AVI(C,q,M)$ with a nonempty $N_{fr}$.
Then $Z$ is invertible if and only if
$\tilde{W}^\tr \tilde{M}\tilde{W}$ is invertible, where
\begin{equation*}
Z = \begin{bmatrix}
M_{BB} & M_{BN_{fr}} & -A^\tr_{\mathcal{A}B}\\
M_{N_{fr}B} & M_{N_{fr}N_{fr}} & -A^\tr_{\mathcal{A}N_{fr}} \\
A_{\mathcal{A}B} & A_{\mathcal{A}N_{fr}} & 0
\end{bmatrix},\quad
\tilde{M} = \begin{bmatrix}
M_{BB} & M_{BN_{fr}} \\
M_{N_{fr}B} & M_{N_{fr}N_{fr}}
\end{bmatrix},\quad
\tilde{W} = \begin{bmatrix}
-A^{-1}_{\mathcal{A}B} A_{\mathcal{A}N_{fr}}\\
I_{N_{fr}}
\end{bmatrix},
\end{equation*}
and $I_{N_{fr}}$ is an identity matrix of size $|N_{fr}| \times |N_{fr}|$.
\label{cor:inv}
\end{corollary}
\begin{proof}
As $A_{\mathcal{A}B}$ is square and invertible,
$\ker \begin{bmatrix} A_{\mathcal{A}B} & A_{\mathcal{A}N_{fr}}\end{bmatrix}
= \im \tilde{W}$. The result follows from \cite[Lemma 3.6]{liu95}.
\end{proof}

We are now ready to present our initial complementary basis and an implicit
extreme point satisfying the sufficient conditions for a ray start.

\begin{proposition}
For a given $\AVI(C,q,M)$, suppose that we have an implicit extreme point
$\bar{z}^0$ and the index sets
$(B,N_l,N_u,N_{fr},\mathcal{A},\bar{\mathcal{A}})$ associated with
$\bar{z}^0$. Then the matrix on the left-hand side of the following system
of equations is invertible if and only if $M$ is invertible on the lineality
space of $C$. Also $z=(z_B,z_{N_{fr}},\bar{z}^0_{N_l},\bar{z}^0_{N_u})$
in a solution to the system of equations
satisfies $z \in \bar{z}^0 + \lin C$, i.e., $z$ is an implicit
extreme point of $C$ by Proposition~\ref{prop:implicit-face} in
Section~\ref{sec:theory}, and $Mz+q \in \aff(N_C(z))$.
\begin{equation*}
\begin{bmatrix}
M_{BB} & M_{BN_{fr}} & -A^\tr_{\mathcal{A}B} & 0 & 0 & 0\\
M_{N_lB} & M_{N_lN_{fr}} & -A^\tr_{\mathcal{A}N_l} & -I_{N_l} & 0 & 0\\
M_{N_uB} & M_{N_uN_{fr}} & -A^\tr_{\mathcal{A}N_u} & 0 & I_{N_u} & 0\\
M_{N_{fr}B} & M_{N_{fr}N_{fr}} & -A^\tr_{\mathcal{A}N_{fr}} & 0 & 0 & 0\\
A_{\mathcal{A}B} & A_{\mathcal{A}N_{fr}} & 0 & 0 & 0 & 0\\
A_{\bar{\mathcal{A}}B} & A_{\bar{\mathcal{A}}N_{fr}} & 0 & 0 & 0 & -I_{\bar{\mathcal{A}}}\\
\end{bmatrix}
\begin{bmatrix}
z_B \\ z_{N_{fr}} \\ \lambda_{\mathcal{A}} \\ w_{N_l} \\ v_{N_u} \\
s_{\bar{\mathcal{A}}}
\end{bmatrix}
=
\begin{bmatrix}
-q_B - M_{BN}\bar{z}^0_N\\
-q_{N_l} - M_{N_lN}\bar{z}^0_N\\
-q_{N_u} - M_{N_uN}\bar{z}^0_N\\
-q_{N_{fr}} - M_{N_{fr}N}\bar{z}^0_N\\
b_{\mathcal{A}} - A_{\mathcal{A}N}\bar{z}^0_N\\
b_{\bar{\mathcal{A}}} - A_{\bar{\mathcal{A}}N}\bar{z}^0_N
\end{bmatrix}.
\end{equation*}
\label{prop:comp}
\end{proposition}
\begin{proof}
The matrix on the left-hand side of the system of equations is invertible if
and only if the matrix $Z$ defined in Corollary~\ref{cor:inv} is invertible.
This is because of the identity submatrices of it, $-I_{N_l},I_{N_u}$, and
$-I_{\bar{\mathcal{A}}}$. Define 
$W = (-A^{-1}_{\mathcal{A}B}A_{\mathcal{A}N_{fr}} \quad
I_{N_{fr}} \quad
0_{N_l} \quad
0_{N_u}
)^\tr$
where $I_{N_{fr}}$ is an identity matrix of size $|N_{fr}| \times |N_{fr}|$, and
$0_{N_l}$ and $0_{N_u}$ are zero matrices of sizes $|N_l| \times |N_{fr}|$ and
$|N_u| \times |N_{fr}|$, respectively. We see that the columns of $W$ is a
basis for the lineality space of $C$. We then have
$W^\tr M W = \tilde{W}^\tr \tilde{M} \tilde{W}$
where $\tilde{W}$ and $\tilde{M}$ are the matrices defined in
Corollary~\ref{cor:inv}. Therefore, the matrix is invertible if and only if
$M$ is invertible on the lineality space of $C$.

We now show that a $z$-part solution $z$ to the system of equations satisfies
$z \in \bar{z}^0 + \lin C$. From the system of equations, we have
\begin{equation*}
z_B = -A_{\mathcal{A}B}^{-1}A_{\mathcal{A}N_{fr}}z_{N_{fr}} +
A^{-1}_{\mathcal{A}B}(b_{\mathcal{A}} - A_{\mathcal{A}N}\bar{z}^0_N).
\end{equation*}

If $z_{N_{fr}}=0$ at a solution, then $z_B=\bar{z}^0_B$.
Therefore, $z=\bar{z}^0$. For $z_{N_{fr}} \neq 0$, we have
$z = \bar{z}^0 + Wz_{N_{fr}}$. As $W$ is a basis for the lineality space of $C$,
it follows that $z \in \bar{z}^0 + \lin C$. As $\bar{z}^0$ is an implicit
extreme point, $z$ is also an implicit extreme point by
Proposition~\ref{prop:implicit-face} in Section~\ref{sec:theory}.

From the first four equations of the given system of equations, we see that
$Mz+q \in \aff (N_C(z))$.
\end{proof}

\section{Worst-case performance comparison:
  \AVI{} vs \MCP{} reformulation}
\label{sec:worst-analysis}

In this section, we introduce the \MCP{} reformulation of an \AVI{}
and analyze worst-case performance of the two formulations in
Sections~\ref{subsec:mcp-reform} and~\ref{subsec:worst-case}, respectively.
We assume that the \MCP{} reformulation is solved using the 
same complementary pivoting method as the \AVI{} formulation is.
Computational results comparing the two formulations are
presented in Section~\ref{sec:exp}, and demonstrate the effectiveness of
working on the different manifold. 
(See Tables~\ref{table:FCLIB}--\ref{tbl:exp-nep} and
Fig.~\ref{fig:path_default-vs-pathavi} in Section~\ref{sec:exp}.)

\subsection{\MCP{} reformulation}
\label{subsec:mcp-reform}

A linear \MCP{} is defined as follows:
for an affine function $F(z)=Mz+q$ and a box constraint
$B_1:=\Pi_{j=1}^n [l_j,u_j]$, $z$ is a solution to the 
$\MCP(B_1,q,M)$ 
if $Mz+q = w-v, z \in B_1, w,v \in \mathbb{R}^n_+,
(z-l)^\tr w = 0$, and $(u-z)^\tr v=0$.

It is well known \cite[page 4]{dirkse96} that an $\AVI(C,q,M)$ can be
reformulated as an $\MCP(B_1 \times B_2, \tilde{q},\tilde{M})$, where
\begin{equation*}
\begin{aligned}
&B_1 = \Pi_{j=1}^n [l_j,u_j], \quad
B_2=\{\lambda \in \mathbb{R}^m \,|\, \lambda \in K^D\},\\
&\tilde{M} = \begin{bmatrix}
M & -A^\tr \\ A & 0
\end{bmatrix}, \quad
\tilde{q} = \begin{bmatrix}
q \\ -b
\end{bmatrix}.
\end{aligned}
\tag{MCP-reform}
\label{eq:mcp-reform}
\end{equation*}
By \cite[Proposition 1.2.1]{facchinei03},
$z^*$ is a solution to the $\AVI(C,q,M)$ if and only if there exists
$\lambda^*$ such that $(z^*,\lambda^*)$ is a solution to the
$\MCP(B_1 \times B_2, \tilde{q},\tilde{M})$. Therefore we can solve an
$\AVI(C,q,M)$ by solving its $\MCP(B_1 \times B_2,\tilde{q},\tilde{M})$
reformulation and vice versa. The solver \PATH{} \cite{dirkse95}, 
one of the most efficient \MCP{} solvers, uses this \MCP{} reformulation
when it takes an \AVI{}.

Although the two formulations are equivalent, they don't have the same
theoretical properties. This is mainly because they look at different feasible
regions, which also results in different \PL{}
manifolds on which the complementary pivoting is performed.
For the $\MCP(B_1 \times B_2,\tilde{q},\tilde{M})$ reformulation,
a \PL{} $(n+m+1)$-manifold
$\mathcal{M}_{B_1 \times B_2}$ is built where the full-dimensional cells are 
defined by the nonempty faces and the normal cones of the set
$B_1\times B_2$, which
doesn't consider the polyhedral constraints $Az - b \in K$.
For the $\AVI(C,q,M)$ formulation, a \PL{} $(n+1)$-manifold
$\mathcal{M}_C$ is constructed  based on the nonempty faces and normal cones
of $C$, which incorporates the polyhedral constraints $Az - b \in K$
explicitly.

\subsection{Worst-case performance analysis}
\label{subsec:worst-case}

In worst-case, the complementary pivoting method may end up having
traversed all the full-dimensional cells of the underlying \PL{}
manifold. As each iteration of the complementary pivoting method corresponds to
the traversal of one full-dimensional cell assuming nondegeneracy or
lexicographic pivoting, the maximum number of iterations is the
total number of the full-dimensional cells, which is finite but could be
exponential in the number of constraints.
Therefore we compare worst-case performance of the two formulations by
counting the number of the full-dimensional cells of the \PL{} manifold
each formulation generates.

By construction, the number of the full-dimensional cells is equivalent to
the number of the nonempty faces of the polyhedral convex set being considered
\cite[page 6]{robinson92}. Thus we count the number of the nonempty faces of
the sets $B_1 \times B_2$ and $C$ each to compare worst-case performance.

Let $\NNF(S)$ denote the number of the nonempty faces of a polyhedral convex
set $S$.
To count the number of the nonempty faces, we start with building blocks
defining a polyhedral convex set: 
intervals $[l,u]$ in $\mathbb{R}$ and linear constraints $a^\tr z - b \in K$.
For a closed interval $[l,u]$ in $\mathbb{R}$, the
number of the nonempty faces is as follows.
\begin{equation}
  \NNF([l,u]) = \left\{
  \begin{array}{ll}
    1 & \text{if } -\infty = l < u = \infty
        \,\text{ or }\, -\infty < l=u < \infty,\\
    2 & \text{if } -\infty = l < u < \infty
        \,\text{ or }\, -\infty < l < u = \infty,\\
    3 & \text{if } -\infty < l < u < \infty.
  \end{array}
\right.
\label{eq:nnf-box}
\end{equation}

For a halfspace or a hyperplane defined by a linear constraint
$a^\tr z - b \in K$ where $a \neq 0$ and $b \in \mathbb{R}$,
the number of the nonempty faces is as follows:
\begin{equation}
  \NNF(\{z \in \mathbb{R}^n\,|\, a^\tr z - b \in K\}) = \left\{
  \begin{array}{ll}
    2 & \text{if } K = \mathbb{R}_+ \,\text{ or }\, K=\mathbb{R}_-,\\
    1 & \text{if } K = \{0\}.
  \end{array}
\right.
\label{eq:nnf-lin}
\end{equation}

Based on \eqref{eq:nnf-box} and \eqref{eq:nnf-lin}, we can compute an
upper bound on the number of the nonempty faces of a polyhedral convex set.

\begin{lemma}
Let $C$ be a polyhedral convex set defined by
$C=\{z \in \mathbb{R}^n \,|\, Az - b \in K, l \le z \le u\}$.
Then
\begin{equation}
\NNF(C) \le \Pi_{j=1}^n \NNF([l_j,u_j]) \times
\Pi_{i=1}^m \NNF(\{z\in\mathbb{R}^n \,|\,A_{i,:}^\tr z - b_i \in K_i\}),
\label{eq:nnf-bnd}
\end{equation}
\label{lemma:nnf-bnd}
where the symbol $\Pi_j$ denotes multiplication over indexed terms.
\end{lemma}
\begin{proof}
Let $C_j = \{z \in \mathbb{R}^n\,|\, z_j \in [l_j,u_j]\}$ for $j=1,\dots,n$ and
$C_{n+i}=\{z \in \mathbb{R}^n\,|\, A_{i,:}^\tr z - b_i \in K_i\}$ for
$i=1,\dots,m$. Then $C=\cap_{i=1}^{n+m} C_i$.
By \cite[Corollary 4.2.15]{robinson15}, $F$ is a face of $C$ if and only if
$F = \cap_{i=1}^{n+m} F_i$ where $F_i$ is a face of $C_i$ for $i=1,\dots,n+m$.
The result follows.
\end{proof}

In Lemma~\ref{lemma:nnf-bnd}, there could be a large gap between $\NNF(C)$
and its upper bound. The upper bound counts all the possible
combinations of the faces of each constraint regardless of their
feasibility. When $C$ has only box constraints, i.e.,
$C=\{z \in \mathbb{R}^n\,|\, l \le z \le u\}$, then equality holds in
\eqref{eq:nnf-bnd}. But, in other cases, the upper bound could be
much larger than $\NNF(C)$ as not every combination corresponds to a
nonempty face of $C$. For example, if
$C=\{z \in \mathbb{R}^2\,|\,z_1+z_2 \ge -1, -z_1+z_2 \ge -1, z_1-z_2 \ge -1,
-z_1-z_2 \ge -1, -1 \le z_1,z_2 \le 1\}$, we have $\NNF(C)=9$. However,
the upper bound is 144. It turns out that there are many infeasible
combinations, i.e., all the combinations having $z_1=-1$ and $z_2=1$.

Using Lemma~\ref{lemma:nnf-bnd}, we prove that maximum number of the
complementary pivoting iterations of the $\AVI(C,q,M)$ formulation
is smaller or equal to the one of its
$\MCP(B_1 \times B_2, \tilde{q}, \tilde{M})$ reformulation.

\begin{proposition}
Let an $\AVI(C,q,M)$ formulation and its
$\MCP(B_1\times B_2,\tilde{q},\tilde{M})$
reformulation defined in \eqref{eq:mcp-reform} be given. Then the number of
the full-dimensional cells of the \PL{} $(n+1)$-manifold
$\mathcal{M}_C$ is less than or equal to the number of the full-dimensional
cells of the \PL{} $(n+m+1)$-manifold $\mathcal{M}_{B_1 \times B_2}$.
\label{prop:plmanifold}
\end{proposition}
\begin{proof}
By \cite[Proposition 4.2.12]{robinson15},
$\NNF(B_1 \times B_2)=\NNF(B_1) \times \NNF(B_2)$.
By applying the same proposition, we have
$\NNF(B_1) = \Pi_{j=1}^n \NNF([l_j,u_j])$ and
$\NNF(B_2)= \Pi_{i=1}^m \NNF([l^\lambda_i,u^\lambda_i])$ where $l^\lambda_i$ and
$u^\lambda_i$ are lower and upper bounds on $\lambda_i$ variable.
Using \eqref{eq:nnf-box} and \eqref{eq:nnf-lin}, we see that
$\Pi_{i=1}^m \NNF(\{z \in \mathbb{R}^n\,|\,A_{i\bullet}^\tr z - b_i \in K_i\})
= \Pi_{i=1}^m \NNF([l_i^\lambda,u_i^\lambda])$. By Lemma~\ref{lemma:nnf-bnd},
the result follows.
\end{proof}

Based on Proposition \ref{prop:plmanifold}, we expect that \PATHAVI{}
will have fewer iterations than \PATH{}, which solves the
\MCP{} reformulation. See computational results in
Sections~\ref{subsec:exp-friction}--\ref{subsec:exp-nash}.

\section{Computational results}\label{sec:exp}

In this section, we present computational results of \PATHAVI{} highlighting
its computational benefits of preserving the problem structure and its
robustness and efficiency compared to \PATH{} 
version 4.7~\cite{dirkse95,path47}, an established solver for \AVI{}s 
which uses the \MCP{} reformulation.
Section~\ref{subsec:exp-preserve}
compares performance of \PATHAVI{} between the original \AVI{} formulation
containing nontrivial lineality space and its equivalent reduced form that
does not contain lines.
Sections~\ref{subsec:exp-friction}--\ref{subsec:exp-nash}
compare performance of \PATHAVI{} and \PATH{} over friction contact problems,
compact sets, and Nash equilibrium problems, respectively.

All experiments were performed on a Linux machine with Intel Xeon(R)
E7-4850 2.00GHz processor and 256GB of memory. \PATHAVI{} 
was compiled using GNU gcc version 4.4.7 and its interfaces were linked to
GAMS. All problem instances were written in GAMS using the EMP syntax for
variational inequalities~\cite{ferris09}. We set the time limit to 1 hour and
iteration limit to $10^5$.

\subsection{Friction contact problem}
\label{subsec:FCP}

Coulomb or dry friction is a ubiquitous phenomenon when mechanical systems interact via contact with each other.
Let two bodies be in contact at one point with $u\coloneqq (u_n, u_t)^\tr\in\R_+\times\R^2$, the relative (or local) velocity between them.
The Coulomb friction phenomenon is described by
\begin{equation}
 \left\{\begin{array}{l@{\quad}l@{}}\text{If}\; u_t = 0&\text{then}\;r\in K_{\mu}\\
   \text{If}\; u_n = 0\;\text{and}\;u_t \neq 0&\text{then}\;r\in \bdry K_{\mu}\;\text{and}\;\exists\alpha\geq0\;\text{such that}\;r_t = -\alpha u_t,
 \end{array}\right.\label{eq:coulomb}
\end{equation}
where $r\coloneqq(r_n, r_t)^\tr $ is the contact force. The friction cone $K_{\mu} \coloneqq \{(t, x)\mid t\in\R_+, x\in\mu tD\}$, with $\mu>0$ and $D$ the unit disk in $\R^2$,
defines the admissible set for the contact force.
There is a host of approaches to computing a solution to this problem, see~\cite{acary2008numerical} for a list of them.
In the following, we use a variational approach that can be traced back to at least~\cite{jean1987dynamics}.
We reformulate~\eqref{eq:coulomb} using normal cone inclusions: $-u_n\in\NC{\R_+}{(r_n)}$ and $ -u_t\in\NC{r_n\mu D}{(r_t)}$.
We cast this problem as a second order LCP (SOLCP):
together with the relations $Mv = Hr + f$ (discretized dynamics) and $u = H^{\tr}v + w$ (transformation from the global velocity to the local one), we obtain
\begin{equation}
 0 \in \begin{pmatrix}M&-H&0\\H^\tr&0&E\\ \bar{H}^\tr&0&E\end{pmatrix} \begin{pmatrix}v\\r\\y\end{pmatrix}
 + \begin{pmatrix} -f\\w\\\bar{w} \end{pmatrix}
 + \;\NC{X}{\begin{pmatrix}v\\r\\y\end{pmatrix}}.\label{avi:fc_vi:main} 
\end{equation}
The matrix $\bar{H}^\tr$ and vector $\bar{w}$ are the same as $H^\tr$ and $w$, except that $(\bar{H}^\tr)_{i\bullet} = 0$ and $\bar{w}_i = 0$ if $i\mod3 \equiv 1$.
Similarly, $E\in\R^{3n_c\times3n_c}$ is defined as $E_{ij} = 1$ if $i = j$ and $i\mod3 \equiv 1$, otherwise $E_{ij} = 0$.
The solution of the SOLCP has to lie in $X\coloneqq\R^{n_{\mathrm{dof}}\cdot n_d}\times K\times K$, $K\coloneqq\Pi_{k=1}^{n_c} \MYK{K}$,
with $n_{\mathrm{dof}}$ the number of degree of freedom, $n_d$ the number of bodies and $n_c$ the number of contacts.
The number of degree of freedom depends on the type of system we consider, i.e. if we have rigid bodies, $n_{\mathrm{dof}} = 6$.
However, if we have deformable bodies, then this number is typically larger and depends on the modeling used.
The cone $K$ is not polyhedral. So we need to approximate $K$ to get an \AVI{} from~\eqref{avi:fc_vi:main}.
To find a solution to the SOLCP, we would have to solve a sequence of \AVI{}s until one of the solutions also satisfies~\eqref{avi:fc_vi:main} up to the tolerance.
However, we focus here on the case where it makes sense to perform a ray start. Hence, we solve the AVI{} that would correspond to the first iteration and with an anisotropic approximation of $K$.
For each contact we construct a finitely representable approximation $\MYK{D}_p$ of the disk $\MYK{\mu} D$. Then the cone $K$ is approximated by $K_p\coloneqq\Pi_k \MYK{K}_p$,
with $\MYK{K}_p\coloneqq\{(t,tx)\mid t\in\R_+, x\in \MYK{D}_p\}$.
Finally, with a slight abuse of notation, we redefine $X\coloneqq \R^{n_{\mathrm{dof}} \cdot n_d}\times K_p\times K_p$ to refer to~\eqref{avi:fc_vi:main} as an $\AVI{}$.
It can be verified that \PATHAVI{} processes the AVI~\eqref{avi:fc_vi:main} if $ w\in\DUAL{(\ker H\cap K)}$ by applying Lemma~\ref{lem:pathAVIext_Lmat}.
It is noteworthy that this condition is exactly the one given in~\cite{klarbring1998existence} for the existence of solution to the SOLCP~\eqref{avi:fc_vi:main}.
If we solely rely on the $\Lmat$ property, we need to assume that $\ker H = \{0\}$, which fails in many instances, for example when a 4-legged chair is in contact with a flat ground.

\subsection{Computational benefits of preserving the problem structure}
\label{subsec:exp-preserve}
The problem data for the following numerical results were obtained from simulations of deformable bodies with the LMGC90~\cite{dubois:hal-00596875} software and using a
solver from \textsc{Siconos}~\cite{acary2007introduction}. In the following, we focus on a simple example where 2 deformable cubes are on top of another.
During the simulation, the number of contacts varies between 80 and 120. The shape of $M$ and $H$ is given in Fig.~\ref{fig:nz-MHW}.
\begin{figure}[t]
  \centering
  \subfloat[Nonzero patterns of $M$]{
    \includegraphics{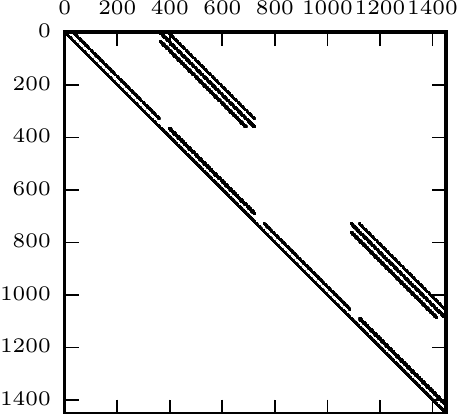}
    \label{fig:nz-M2}
  }
  \subfloat[Nonzero patterns of $H$]{
    \includegraphics{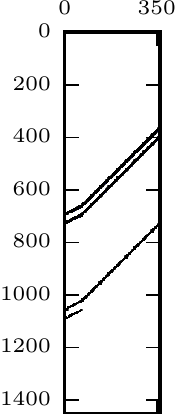}
    \label{fig:nz-H}
  }
  \subfloat[Nonzero patterns of $W$]{
   \includegraphics{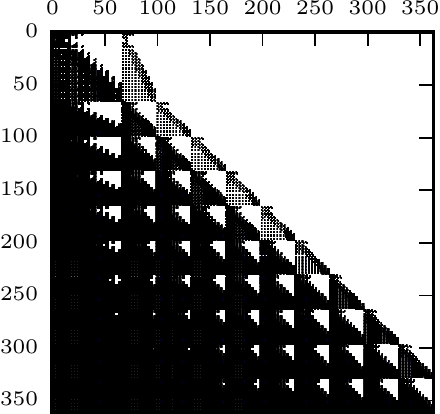}
    \label{fig:nz-W}
  }
  \caption{Nonzero patterns of the matrices $M$ (size: $1452\times 1452$, nnz: $11330$), $H$ (size: $1452\times 363$, nnz: $1747$) and $W\coloneqq H^TM^{-1}H$ (size: $363\times 363$, nnz:  $56770$).}
\label{fig:nz-MHW}
\end{figure}
It is noteworthy that if we have to remove the lineality space, that is compute $W$, then all the structure of the problem is destroyed
(see Fig. \ref{fig:nz-W}):
if we remove the lineality space, the number of nonzero elements increases by a factor of $5$!
It is expected that the linear algebra computations will be more expensive in the reduced space than in the original one because of this large increase of
nonzero entries.
This has been verified on instances that have the same kind of structure as the matrices depicted in Fig.~\ref{fig:nz-MHW}.
\begin{figure}[htb]
 \centering
 \includegraphics{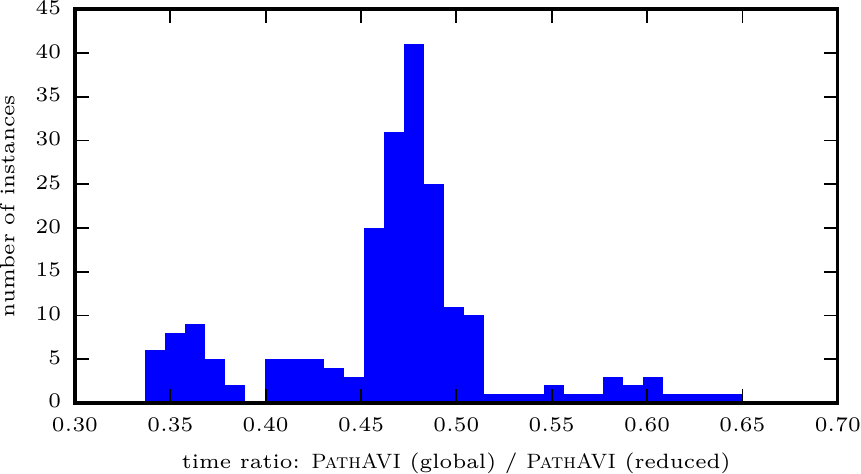}
 \caption{Comparison in terms of speed between the resolution in the original space and the reduced one. The number of iteration was the same for the 209 instances.}
 \label{fig:reduced_is_bad}
\end{figure}
As shown on Fig.~\ref{fig:reduced_is_bad}, \PATHAVI{} working in the original space is always faster and most of the time is at least twice as fast as \PATHAVI{} working in the reduced space.
The time in the reduced space does not take into account the transformation of the problem data (computation of the $W$ matrix).

\subsection{Multibody friction contact problems}
\label{subsec:exp-friction}

When the bodies are rigid, it is common in the contact mechanic community to eliminate the velocity $v$.
The problem is formulated in a reduced space $K_p\times K_p$ (defined in
Section~\ref{subsec:FCP}) and the \AVI{} is
\begin{equation}
 0\in \begin{pmatrix}W&E\\\bar{W}&E\end{pmatrix} \begin{pmatrix} r\\y \end{pmatrix} + \begin{pmatrix}
  \omega \\ \bar{\omega}
 \end{pmatrix} + \NC{K_p\times K_p}{\begin{pmatrix} r\\y \end{pmatrix}},
 \label{eq:fc_condensed}
\end{equation}
where $W\coloneqq H^{\tr}M^{-1}H$ and $\bar{W}\coloneqq \bar{H}^{\tr}M^{-1}H$, $\omega\coloneqq w + H^{\tr}M^{-1} f$ and $\bar{\omega}\coloneqq \bar{w} + \bar{H}^{\tr}M^{-1} f$.
The lineality space is then trivial in this formulation.

We present computational results using the problem data ($W$, $\mu$ and $q$) from the FCLIB collection\footnote{The collection of problem can be freely downloaded by visiting \url{http://fclib.gforge.inria.fr}}~\cite{acary:hal-00945820},
which aims at providing challenging instances of the friction contact problem. Since we know that \PATHAVI{} can find a solution to any of these examples, the tolerance is set to a low value: $\sqrt{N}\cdot 10^{-9}$, where $N$
is the number of contacts. This value is lower than the default tolerance of \PATH{} (that is already considered quite demanding).
First a few observations: \PATH{} fails to perform a ray start on all the examples, due to the fact that $\ker W$ is nontrivial. 
The results are summarized in Table~\ref{table:FCLIB} and show that \PATHAVI{} is more robust than \PATH{}. 
\begin{table}[htb]
\centering
\caption{Statistics for 4579 friction contact problems of the form~\protect\eqref{eq:fc_condensed}.}
\begin{tabular}{|l|r|r|r|r|r|}
 \hline
 \multirow{2}{*}{Solver/profile} & \multirow{2}{*}{\# Failed} & 
                             \multicolumn{4}{|c|}{Failure type} \\\cline{3-6}
                  &      & Solver error & Stalled & Time & Iteration \\\hline
pathavi/UMFPACK   &    0 &            0 &       0 &    0 &         0 \\
pathavi/default   &   17 &            0 &       0 &    0 &        17 \\
pathavi/LUSOL-blu &    4 &            0 &       0 &    0 &         4 \\
path/default      & 2060 &          535 &    1525 &    0 &         0 \\
path/no crash     &  108 &          101 &       0 &    6 &         1 \\
\hline
\end{tabular}
\label{table:FCLIB}
\end{table}
\PATHAVI{} with the linear algebra package UMFPACK (``pathavi/UMFPACK'') solves all instances, and changing the linear algebra routines to LUSOL (``pathavi/default'')
leads to a small number of failures. This number can be reduced by using the block-LU updates~\cite{eldersveld92} (``pathavi/LUSOL-blu'').
The default behavior of \PATH{} (``path/default'') leads to many failures: the crash method is inappropriate for such models.
However, even without the crash procedure (``path/no crash''), \PATH{} still fails on more instances than \PATHAVI{}.
The different failure types have the following meaning: ``Solver error'' means that the first basis matrix could not be factorized, despite the use of artificial variables to
overcome the rank deficiency. ``Stalled'' means that a solver possibly tried various strategies but failed to generate solution of the requested accuracy and consequently gave up.
Note that this never occurred with \PATHAVI{} on this set of problems.
``Time'' (or ``Iteration'') signals that the time (or iteration) limit has been reached.
\begin{figure}[t]
  \centering
  \subfloat[Number of iterations]{
    \includegraphics[scale=.75]{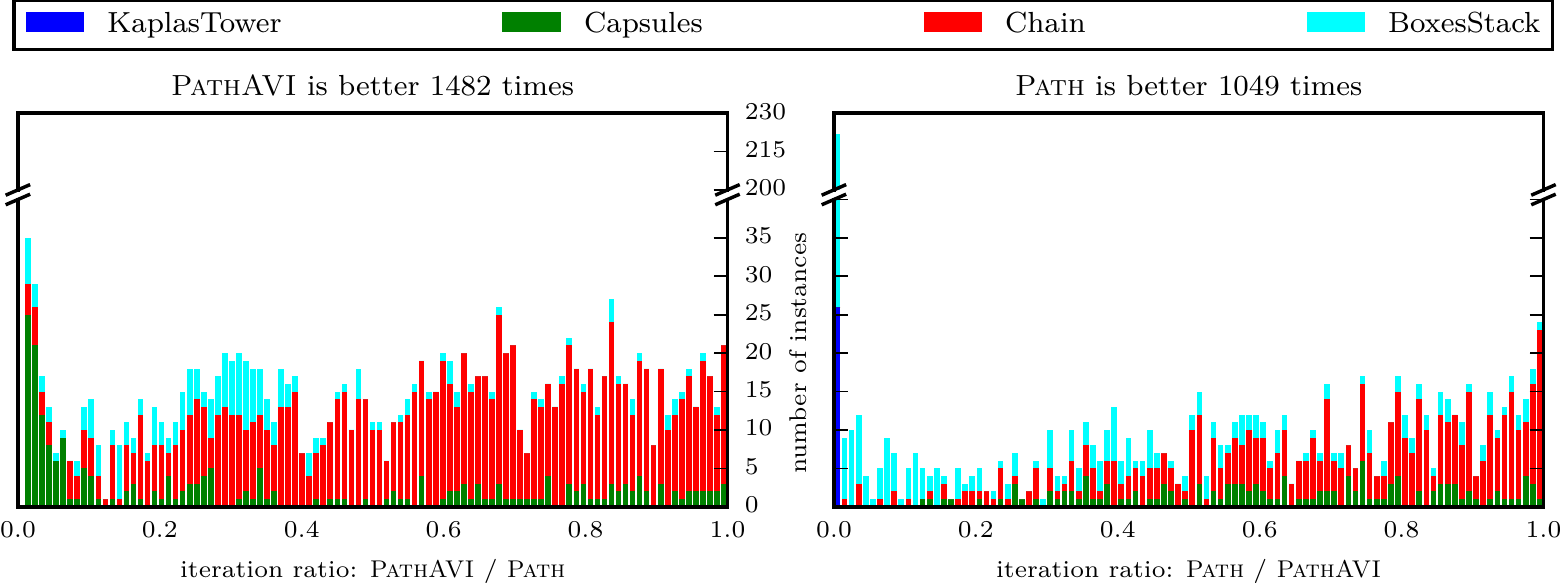}
    \label{fig:path_default-vs-pathavi--miter}
  }\\[.2mm]
  \subfloat[Time]{
    \includegraphics[scale=.75]{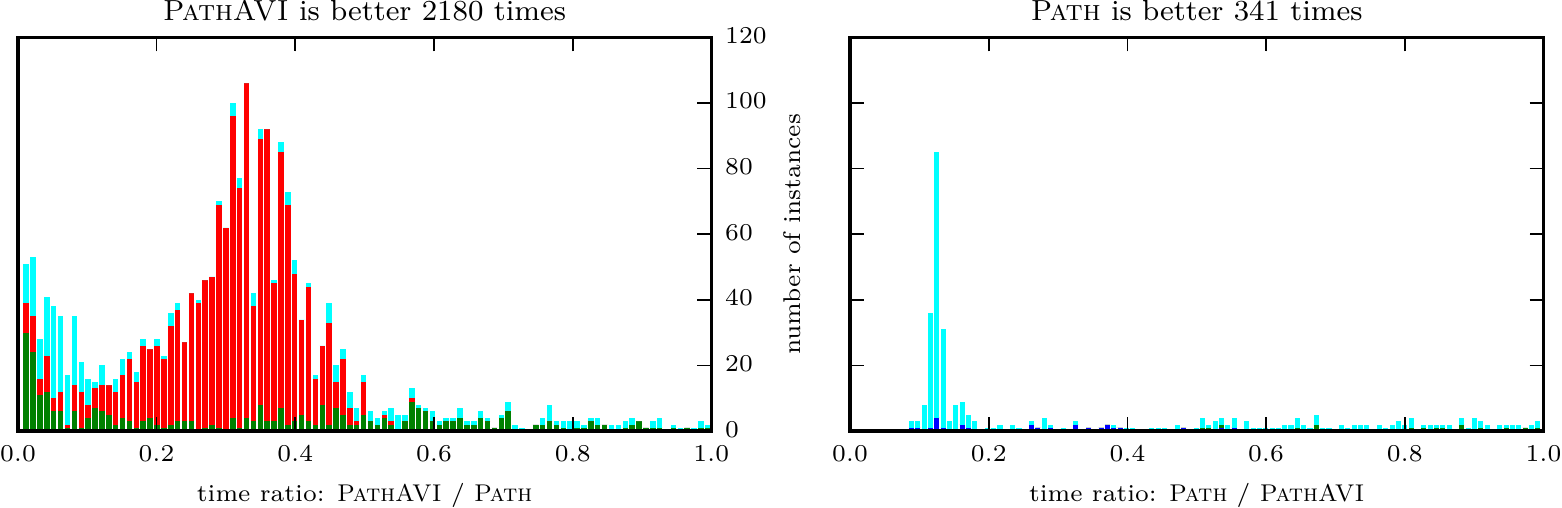}
    \label{fig:path_default-vs-pathavi--time}
  }
  \caption{Comparison between \PATH{} and \PATHAVI{}}
\label{fig:path_default-vs-pathavi}
\end{figure}
Due to space constraints, we further compare only \PATH{} and \PATHAVI{} with their default settings.
First the performance of both solvers in terms of number of iterations is displayed on Fig.~\ref{fig:path_default-vs-pathavi--miter}.
The left plot represents the ratio of the number of iterations when \PATHAVI{} has a fewer number. For the right plot, it is when \PATH{}
solved with a fewer number. Overall, \PATHAVI{} is better than \PATH{}. The spike on the right plot, when \PATH{} finds the solution with
a small number of iterations compared to \PATHAVI{}, is explained by the fact that the crash procedure performed well in those cases.
However, as it can be seen in Fig.~\ref{fig:path_default-vs-pathavi--time}, this does not imply that \PATH{} is faster, since those crash iterations
can be expensive. If speed is the metric, then \PATH{} is better in less than 10\% of all the instances. The previous spike
illustrates that the crash method can find a very good starting point in some instances. This feature of a Newton-based method, which either finds a solution quickly or fails, has already been witnessed when solving friction contact problems.
Finally note that when \PATHAVI{} is the faster solver, it usually finds a solution in less than half the time of \PATH{}.


\subsection{\AVI{}s over compact sets}
\label{subsec:exp-compact}

One strong implication of Theorem~\ref{thm:class} is that when $C$ is compact
(so that $\rec C=\{0\}$)
\PATHAVI{} can process an $\AVI(C,q,M)$ with arbitrary $M$ and $q$. 
In contrast, this does not hold for the \MCP{} reformulation
as the underlying feasible region of it
may not be compact although $C$ is compact. This is because whenever the
\AVI{} contains polyhedral constraints the associated $\lambda$ variables
in the \MCP{} reformulation are unbounded.

We construct 5 \AVI{} instances by taking compact feasible regions from
\cite{maros98} having finite lower and upper bounds and by randomly generating
$M$ and $q$ such that the resultant \AVI{} has an $M$ with negative eigenvalues.

Table~\ref{tbl:exp-compact} presents some computational results. As expected,
\PATHAVI{} is able to solve all the instances, whereas \PATH{} fails to solve
three of them. Also on the two problem instances where both solvers are
able to solve, \PATHAVI{} shows 10--30 times fewer iterations, and
a similarly decreased elapsed time. These properties hold for a wide selection
of instances and the above table is just provided for expository purposes.

\begin{table}[t]
\centering
\caption{Performance of \PATHAVI{} and \PATH{} over compact sets}
\begin{tabular}{|r|r|r|r|r|r|r|}
\hline
\multicolumn{1}{|c|}{\multirow{2}{*}{Name}} & 
\multicolumn{1}{|c|}{\multirow{2}{*}{(\#constrs,\#vars)}} & 
\multicolumn{1}{|c|}{\multirow{2}{*}{(nnz(A),nnz(M))}} & 
\multicolumn{2}{c|}{Number of iterations} & 
\multicolumn{2}{c|}{Elapsed time (secs)} \\\cline{4-7}
 & & & \multicolumn{1}{c|}{\PATHAVI{}} & \multicolumn{1}{c|}{\PATH{}} 
     & \multicolumn{1}{c|}{\PATHAVI{}} & \multicolumn{1}{c|}{\PATH} \\\hline
CVXQP1\_M &  (500, 1000) &  (2495,  999) &  3119 & fail &   0.459 &    fail\\
CVXQP2\_M &  (250, 1000) &  (1746,  999) & 33835 & fail &   2.827 &    fail\\
CVXQP3\_M &  (750, 1000) &  (3244,  999) &   360 & 3603 &   0.105 &   1.992\\
CONT-050  & (2401, 2597) & (14597, 6407) &    11 &  382 &   2.753 & 272.429\\
CONT-100  & (9801,10197) & (59197,98875) &     3 & fail & 174.267 &    fail\\
\hline
\end{tabular}
\label{tbl:exp-compact}
\end{table}

\subsection{Nash equilibrium problems}
\label{subsec:exp-nash}

Another application of \AVI{}s is to Nash equilibrium problems.
In a Nash equilibrium problem, there are multiple agents each of which
minimizing its own objective function, and each agent's objective function
not only depends on the agent's decision but also other agents' decisions. 
For example, a typical Nash equilibrium problem computes a solution
satisfying
\begin{equation}
x_i^* \in \argmin_{x_i \in X_i}\, h_i(x_i,x^*_{-i}), \quad 
\text{for } i=1,\dots,N.
\tag{NEP}
\label{eq:nep}
\end{equation}
where we note that each $i$th agent's objective function $h_i$ takes
its own decision, denoted by $x_i$, and other agents' decisions, denoted
by $x_{-i}$.

We generated 6 instances of Nash equilibrium problems, where
each $X_i$ is a polyhedral convex set and $h_i$ is continuously 
differentiable in $x$ and convex quadratic in $x_i$ for each fixed $x_{-i}$.
Specifically, $h_i$ takes the following form:
\begin{equation*}
h_i(x_i,x_{-i}) = \frac{1}{2}x_i^\tr Q_i x_i + x_i^\tr Q_{-i}x_{-i}
+ c_i^\tr x_i + d_i^\tr x_{-i}.
\end{equation*}
where $Q_i$ is symmetric positive definite.

In this case, $x$ is a solution to \eqref{eq:nep}
if and only if it is a solution to the $\AVI(C,q,M)$ where
$Mx+q = (\nabla_{x_i} h_i(x))_{i=1}^N$ and $C=\Pi_{i=1}^N X_i$.
The number of agents ranges from 10 to 300. 

\begin{table}[t]
\centering
\caption{Performance of \PATHAVI{} and \PATH{} over the NEPs}
\subfloat[Statistics of the NEPs]{
  \begin{tabular}{|r|r|r|}
    \hline
    \multicolumn{1}{|c|}{Name} & 
    \multicolumn{1}{c|}{(\#constrs,\#vars)} & 
    \multicolumn{1}{c|}{(nnz(A),nnz(M))} \\\hline
    vimod1 & ( 554,1138) & (4744,22577) \\
    vimod2 & ( 910,1723) & (7935,46137) \\
    vimod3 & (1101,2226) & (9117,67634) \\
    vimod4 & ( 870,1828) & (62056,154332) \\
    vimod5 & (1327,2586) & (133527,274004) \\
    vimod6 & (2210,4359) & (207408,417810) \\
    \hline
\end{tabular}
}\\
\subfloat[\# Iterations and elapsed time of \PATHAVI{} and \PATH{} on the NEPs]{
  \begin{tabular}{|r|r|r|r|r|r|r|}
    \hline
    \multicolumn{1}{|c|}{\multirow{3}{*}{Name}} &
    \multicolumn{3}{c|}{Number of iterations} & 
    \multicolumn{3}{c|}{Elapsed time (secs)} \\\cline{2-7}
    & \multicolumn{1}{|c|}{\multirow{2}{*}{\PATHAVI{}}} 
    & \multicolumn{1}{c|}{\multirow{2}{*}{\PATH{}}} 
    & \multicolumn{1}{c|}{\PATHAVI{}/} 
    & \multicolumn{1}{c|}{\multirow{2}{*}{\PATHAVI{}}}
    & \multicolumn{1}{c|}{\multirow{2}{*}{\PATH{}}} 
    & \multicolumn{1}{c|}{\PATHAVI{}/} \\
    & & & \multicolumn{1}{c|}{UMFPACK}
    & & & \multicolumn{1}{c|}{UMFPACK} \\\hline
    vimod1 &  367 & 2087 &  367 &    0.372 &    4.129 &   0.437\\
    vimod2 &  319 & 3570 &  319 &    1.098 &   24.134 &   0.645\\
    vimod3 &  590 & 4278 &  590 &    3.208 &   60.553 &   1.639\\
    vimod4 & 1343 & 6146 & 1343 &  127.194 &   66.427 &  18.319\\
    vimod5 & 2167 & 2768 & 2167 &  327.970 &  325.558 &  40.285\\
    vimod6 & 3522 & 4222 & 3522 & 2341.193 & 1841.642 & 109.960\\\hline
  \end{tabular}
}
\label{tbl:exp-nep}
\end{table}

Table~\ref{tbl:exp-nep} presents performance of \PATHAVI{} and \PATH{} over
the NEPs.
The number of iterations of \PATHAVI{} is up to 11 times fewer than \PATH{}.
Elapsed time shows similar results except for the last three instances.
In those instances, LUSOL has a great difficulty in computing
\PATHAVI{}'s intermediate basis matrices. If we change the linear algebra
engine to UMFPACK, the computation time significantly reduces. 
Regarding \PATH{}'s performance on the last three instances, we would like to
point out that the proximal perturbation technique of \PATH{},
which solves a sequence of perturbed \MCP{}s by adding positive diagonal
elements $\epsilon_k I$ with $\epsilon_k \rightarrow 0$ as
$k \rightarrow \infty$ to the matrix $\tilde{M}$ in \eqref{eq:mcp-reform},
plays a significant role in its performance. 
Adding positive diagonals changes the elimination sequence and makes
linear algebra computations much faster and more stable. 
When we turn off the proximal perturbation, \PATH{} either gets much slower 
than \PATHAVI{} or fails to solve the instance. 

\section{Conclusions}\label{sec:conclusion}

We have presented \PATHAVI{}, a structure-preserving pivotal method for
affine variational inequalities. Compared to existing methods,
\PATHAVI{} can process an \AVI{} without applying any reduction or
transformation to the problem data even if the underlying feasible region
contains lines.
\PATHAVI{} can process some newly generated problem classes from
applications in friction contact as well as the existing
problem class ($L$-matrices \cite{cao96}).
A computational method for finding a point satisfying sufficient conditions
for a ray start is detailed.
Through worst-case analysis, we have shown that exploiting polyhedral
structure for solving affine variational inequalities is expected to show
better performance than using a mixed complementarity problem reformulation.
Computational results over friction contact and Nash equilibrium
problems demonstrate that \PATHAVI{} compares favorably with \PATH{} in terms of
robustness and efficiency.




\section*{Acknowledgements}
This work is supported in part by the Air Force Office of Scientific Research and the Department of Energy. The authors are grateful to Steven Dirkse and Todd Munson for comments and suggestions leading to improved computational performance.

\newpage
\bibliographystyle{siam}       
\bibliography{pathavi}   

\section*{Appendix}
\label{sec:appendix}

\begin{lemma}[Theorem~4.4~\cite{cao96}]\label{lem:aux_var_cst}
 Consider an $\AVI(C,q,M)$ and let $M$ be semimonotone with respect to $\rec C$ and invertible on the lineality space of $C$.
 Suppose that an unbounded ray occurs.
 Then the value of the auxiliary variable $t$ is constant on that ray and $\Delta z$, the variation in $z$ is nonzero and satisfies
 \begin{equation}
  \Delta z \in \rec C,\qquad M\Delta z \in \DUAL{(\rec C)}, \qquad\text{and}\; \Delta z^\tr M\Delta z = 0.\label{eq:delta_z_ray}
 \end{equation}
\end{lemma}
\begin{proof}
 The fact that $\Delta t = 0$ and that $\Delta z$ is a solution to~\eqref{eq:delta_z_ray} follows from the first part of the proof of Theorem~4.4 in~\cite{cao96}.
 To see that the direction $\Delta \reds{z}$ is nonzero, we proceed by contradiction: at the current iterate $(\MYK{x}, \MYK{t})$ we have
\begin{align}
 G_{\reds{C}}(\MYK{x}, \MYK{t}) &= \reds{M}\MYK{z} + \reds{q} + \MYK{x} - \MYK{z} - \MYK{t} r = 0.\label{eq:fc_vi:Gk}
 \intertext{Let $\MYKp{x}$ belong to the unbounded ray and suppose that $\Delta z = 0$:}
 G_{\reds{C}}(\MYKp{x}, \MYK{t}) &=  \reds{M}\MYK{z} + \reds{q} + \MYKp{x} - \MYK{z} - \MYK{t} r = 0.\label{eq:fc_vi:Gkp1}
\end{align}
It immediately follows that $\MYKp{x} = \MYK{x}$.
\end{proof}

%
%

\begin{table}[b]
\centering
\caption{Index sets and a basis matrix describing a basic solution $z$ of an LP
  problem. Assume that $z \in \mathbb{R}^n, A \in \mathbb{R}^{m \times n}$, and
  $b \in \mathbb{R}^m$}
\label{tbl:basis}
\begin{tabular}{l}
  $B \cup N_l \cup N_u \cup N_{fr} = \{1,\dots,n\}$ and
  $B, N_l, N_u, \text{ and } N_{fr} \text{ are mutually exclusive.}$\\
  $B := \text{ a set of basic variables indices}$\\
  $N_l := \text{ a set of nonbasic variables indices at their finite lower bounds}$\\
  $N_u := \text{ a set of nonbasic variables indices at their finite upper bounds}$\\
  $N_{fr} := \text{ a set of nonbasic free variables indices}$\\
  $\mathcal{A} \cup \bar{\mathcal{A}} = \{1,\dots,m\} \quad \text{with}
  \quad \mathcal{A} \cap \bar{\mathcal{A}} = \varnothing$\\
  $\mathcal{A} := \text{ a set of active constraints indices, i.e., }
  A_{\mathcal{A}\bullet}z=b_{\mathcal{A}}$\\
  $\bar{\mathcal{A}} := \text{ a set of inactive constraints indices}$\\
  $\mathbf{B} =
  \begin{bmatrix}
    A_{\mathcal{A}B} & 0\\
    A_{\bar{\mathcal{A}}B} & \pm I_{\bar{\mathcal{A}}}
  \end{bmatrix} \text{ is an invertible basis matrix where }
  I_{\bar{\mathcal{A}}} \text{ is an identity matrix of size }
  |\bar{\mathcal{A}}| \times |\bar{\mathcal{A}}|$\\
  $z_B = A_{\mathcal{A}B}^{-1}\left(b_{\mathcal{A}}-A_{\mathcal{A}N}z_N\right),
  \, z_l = l_{N_l},\,z_u = u_{N_u},\, z_{N_{fr}} = 0,\,N=N_l\cup N_u\cup N_{fr}$
\end{tabular}
\end{table}

\begin{algorithm}[b]
\caption{Pivoting to make as many nonbasic free variables as basic variables}
\begin{algorithmic}[1]
\Require a basic feasible solution $z^0$ and its index sets 
$(B^0,N^0_l,N^0_u,N^0_{fr},\mathcal{A}^0,\bar{\mathcal{A}}^0)$
\Ensure a basic feasible solution $\bar{z}^0$ and its index sets 
$(B,N_l,N_u,N_{fr},\mathcal{A},\bar{\mathcal{A}})$
\State Set $\bar{z}^0 \leftarrow z^0$.
\State Set $(B,N_l,N_u,N_{fr},\mathcal{A},\bar{\mathcal{A}}) \leftarrow 
(B^0,N^0_l,N^0_u,N^0_{fr},\mathcal{A}^0,\bar{\mathcal{A}}^0)$.
\State Set $\texttt{changed} \leftarrow \texttt{true}$.
\While{\texttt{changed} is \texttt{true}}
\State Set $\texttt{changed} \leftarrow \texttt{false}$.
\For{each $j \in N_{fr}$}
\State Do a ratio test on the nonbasic column $j$ over basic variables that
are not free variables.
\If{the ratio is finite}
\State Pivot in the $j$th column into basis.
\State Update $\bar{z}^0$ and its index sets $(B,N_l,N_u,N_{fr},\mathcal{A},\bar{\mathcal{A}})$.\Comment{$|N_{fr}| \leftarrow |N_{fr}|-1$}
\State Set $\texttt{changed} \leftarrow \texttt{true}$.
\EndIf
\EndFor
\EndWhile
\Return $\bar{z}^0$ and its index sets $(B,N_l,N_u,N_{fr},\mathcal{A},\bar{\mathcal{A}})$
\end{algorithmic}
\label{alg:pivoting}
\end{algorithm}

\end{document}